\documentclass[oneside, 10pt]{article}

\usepackage{amssymb, amsmath}
\usepackage{euscript}
\usepackage{mathrsfs}
\usepackage{xypic}
\usepackage{ntheorem}
\usepackage{enumitem}
\usepackage{url, hyperref}
\usepackage{tikz, tikz-cd}
\usepackage{stmaryrd}

\theoremseparator{.}
\theoremprework{\smallskip\smallskip}
\newtheorem{thm}{Theorem}[section]
\theoremprework{\smallskip\smallskip}
\newtheorem{cor}[thm]{Corollary}
\theoremprework{\smallskip\smallskip}
\newtheorem{lem}[thm]{Lemma}
\theoremprework{\smallskip\smallskip}
\newtheorem{prop}[thm]{Proposition}
\theoremprework{\smallskip\smallskip}
\theorembodyfont{\upshape}
\newtheorem{defn}[thm]{Definition}
\theoremprework{\smallskip\smallskip}
\newtheorem{rmk}[thm]{Remark}
\theoremprework{\smallskip\smallskip}
\newtheorem{ex}[thm]{Example}

\newenvironment{proof}[1][Proof]{\begin{trivlist}\item[]{\textsc{#1}.\
}}{\nolinebreak $\Box$ \end{trivlist}} 

\newcommand{\stack}[2]{\begin{subarray}{c} #1 \\ #2 \end{subarray}}

\newcommand{\mb}{\mathbb}
\newcommand{\mc}{\mathcal}
\newcommand{\mr}{\mathrm}
\newcommand{\msf}{\mathsf}
\DeclareMathAlphabet{\mathpzc}{OT1}{pzc}{m}{it}
\DeclareMathAlphabet{\mf}{U}{euf}{m}{n}
\DeclareMathAlphabet{\ms}{U}{rsfso}{m}{n}

\newcommand{\hra}{\hookrightarrow}
\newcommand{\tensor}{\otimes}
\newcommand{\ul}{\underline}
\renewcommand{\bar}{\overline}
\newcommand{\sfrac}[2]{{\textstyle\frac{#1}{#2}}}
\renewcommand{\tilde}{\widetilde}

\newcommand{\DM}{{\msf{DM}}}
\newcommand{\QS}{{\msf{QS}}}
\newcommand{\Sch}{{\msf{Sch}}}
\newcommand{\St}{{\msf{St}}}
\newcommand{\Var}{{\msf{Var}}}

\renewcommand{\phi}{\varphi}

\newcommand{\YY}{\mf Y}   
\newcommand{\WW}{\mf W}      
\newcommand{\XX}{\mf X}
\newcommand{\ZZ}{\mf Z}      

\newcommand{\pmat}[1]{\begin{pmatrix}#1\end{pmatrix}}
\newcommand{\smat}[1]{\bigl(\begin{smallmatrix}#1\end{smallmatrix}\bigr)}

\newcommand{\dmo}[2]{\DeclareMathOperator{#1}{#2}}
\dmo{\id}{id}
\dmo{\im}{im}

\dmo{\uAut}{\ul{Aut}}
\dmo{\uHom}{\ul{Hom}}
\dmo{\spec}{Spec}
\dmo{\Aut}{Aut}
\dmo{\B}{B}
\dmo{\BGL}{BGL}
\dmo{\End}{End}
\dmo{\Gal}{Gal}
\dmo{\GL}{GL}
\dmo{\Hom}{Hom}
\dmo{\Epi}{Epi}
\dmo{\Inj}{Inj}
\dmo{\Spec}{Spec}
\dmo{\Stab}{Stab}
\dmo{\In}{\mr{I}}
\dmo{\Iss}{\mr{I^{ss}}}
\dmo{\Iu}{\mr{I^u}}
\dmo{\D}{\mr{D}}
\dmo{\DZI}{\mr{DZI}}
\dmo{\ZI}{\mr{ZI}}
\dmo{\NZI}{\mr{NZI}}
\dmo{\ZoI}{\mr{Z^0I}}
\dmo{\ZssoI}{\mr{Z^0_{ss}I}}
\dmo{\DZssI}{\mr{DZ_{ss}I}}
\dmo{\ZssI}{\mr{Z_{ss}I}}
\dmo{\ZuI}{\mr{Z_{u}I}}

\renewcommand{\ss}{{\text{\it ss}}}
\newcommand{\longiso}{\stackrel{\simeq}{\longrightarrow}}
\renewcommand{\setminus}{\mathbin{\mathchoice{\mysetminusD}{\mysetminusT}{\mysetminusS}{\mysetminusSS}}}
	\newcommand{\mysetminusD}{\hbox{\tikz{\draw[line width=0.6pt,line cap=round] (3pt,0) -- (0,6pt);}}}
	\newcommand{\mysetminusT}{\mysetminusD}
	\newcommand{\mysetminusS}{\hbox{\tikz{\draw[line width=0.45pt,line cap=round] (2pt,0) -- (0,4pt);}}}
	\newcommand{\mysetminusSS}{\hbox{\tikz{\draw[line
                width=0.4pt,line cap=round] (1.5pt,0) -- (0,3pt);}}}

\newcommand{\cartesian}{ \ar @{} [dr] |{\square}}
\newcommand{\diag}{\xymatrix}

\title{The Eigenvalue Spectrum of the Inertia Operator}
\author{Kai Behrend, Pooya Ronagh}
\date{\today}

\begin{document}

\maketitle

\begin{abstract}
We view the inertia construction of algebraic stacks as an operator on the Grothendieck groups of various categories of algebraic stacks. We show that the inertia operator is locally finite and diagonalizable. This is proved for the Grothendieck group of Deligne-Mumford stacks over a base scheme and the category of quasi-split Artin stacks defined over a field of characteristic zero. 
Motivated by the quasi-splitness condition, in \cite{BRHall} we consider the inertia operator of the Hall algebra of algebroids, and applications of it in generalized Donaldson-Thomas theory. 
\end{abstract}

\tableofcontents

\section{Introduction} 

Let $B$ be a unital noetherian commutative ring and $\Sch/B$, the big \'etale site of schemes of finite type over $B$. $\St/B$ denotes the category of algebraic stacks of finite type over $\Sch/B$ with affine diagonals. This in particular means that for every algebraic stack $\XX$ and an $S$-point $s: S \to \XX$, the sheaf of automorphisms $\uAut(s) \to S$ is an affine $S$-group scheme.

The Grothendieck group of $\Sch/B$, denoted by $K(\Sch/B)$ is the free abelian group of isomorphism classes of such schemes, modulo the scissor relations, 
$[X] = [Z] + [X \setminus Z],$ for $Z \subset X$ a closed subscheme and
equipped with structure of a commutative unital ring according to the fibre product in $\Sch/B$,
$[X] \cdot [Y] = [X \times Y].$
We will always tensor with $\mb Q$. 

The Grothendieck ring of the category $\St/B$ of algebraic stacks over $\Sch/B$, is defined similar to above: $K(\St/B)$ is the $\mb Q$-vector space generated by isomorphism classes of algebraic stacks modulo similar relations; i.e. for any closed immersion $\ZZ \hra \XX$ of algebraic stacks we have $[\XX] = [\ZZ] + [\XX \setminus \ZZ]$. And the fiber product over the base category turns $K(\St/B)$ into a commutative ring. Hence for any algebraic stack $\mf Y$, isomorphic to a fiber product $\XX \times \ZZ$, we have $[\YY] = [\XX][\ZZ]$.
Moreover, $K(\St/B)$ is a unital associative $K(\Sch/B)$-algebra in the obvious way. 

For any algebraic stack, $\XX$, the inertia stack $\In\XX$ is the fiber product 
$$\diag{
\cartesian \In \XX \ar[d] \ar[r] & \XX \ar[d]^{\Delta} \\
\XX \ar[r]_{\Delta} & \XX \times \XX}$$
where $\Delta$ is the diagonal morphism. $\In \XX$ is isomorphic to the stack of objects $(x, f)$, where $x$ is an object of $\XX$ and $f: x \to x$ is an automorphism of it. Here a morphism $h: (x, f) \to (y, g)$ is an arrow $h: x \to y$ in $\XX$ satisfying $g \o h = h \o f$. 

This construction respects the equivalence and scissor relations and is \linebreak $K(\Sch/B)$-linear, hence inducing a well-defined \emph{inertia operator} on $K(\St/B)$ as a $K(\Sch/B)$-algebra, and in particular an \emph{inertia endomorphism} on the $K(\Sch/B)$-module $K(\St/B)$. 

Other important variants of the inertia operator are the semisimple and unipotent inertia operators, respectively $I^{ss}$ and $I^u$, which are defined in \S\ref{sec:semisimple-inertia-definition}.

We always assume that $K$ is tensored with $\mb Q$. Our main results on local finiteness and diagonalization of the inertia endomorphism are as follows. Here $q= [\mb A^1]$ is our notation for the class of the affine line. 
\begin{enumerate} 
\item Corollary \ref{cor:diagonalization-DM}: 
In the case of Deligne-Mumford stacks the operator $I$ is diagonalizable as a $K(\Sch / B)$-linear endomorphism, and the eigenvalue spectrum of it is equal to $\mb N$, the set of positive integers. 
\item Corollary \ref{cor:unipotent-inerta}: 
In the case of Artin stacks, the unipotent inertia $I^u$ is diagonalizable on $K(\St)[q^{-1}, \{(q^k - 1)^{-1}: k \geq 1\}]$ and the eigenvalue spectrum of it is the set $\{q^k: k \geq 0\}$ of all power of $q$. 
\item Theorem \ref{thm:main}: 
In the case of quasi-split stacks, the operator $I$ is diagonalizable as a $\mb Q(q)$-linear endomorphism and the eigenvalue spectrum of it is the set of all polynomials of the form 
$$n q^u \prod_{i=1}^k (q^{r_i} -1).$$
\item Theorem \ref{thm:semisimple-main}: 
In the case of quasi-split stacks the endormorphism $I^{ss}$ is diagonalizable as a $\mb Q(q)$-linear operator and 
the eigenvalue spectrum of it is the set of all polynomials of the form 
$$n \prod_{i=1}^k (q^{r_i} -1).$$
\end{enumerate} 
This paper is organized as follows. \S\ref{ch:groups} contains several results needed for stratification of group schemes in the rest of the paper. The reader may skip this section and refer to it as needed in the proofs of the results in other sections.  
In \S\ref{ch:Gpd} we consider the simplest case; i.e. the category of groupoids. The arguments in this section illustrate the main ideas behind the results of this paper. In \S\ref{ch:DM} we consider the full subcategory $\DM$ of $\St$, of Deligne-Mumford stacks. For the cases of Artin stacks in \S\ref{ch:artin} and \S\ref{ch:quasi-split}, the base category would be the category of varieties over the spectrum of an algebraically closed field in characteristic zero.

\section{Stratification of group schemes}
\label{ch:groups}

\subsection{Connected component of unity} 

By a \emph{group space} $G \to X$, we mean a group object in the category $\St/B$. In this section we will see that we can always stratify such objects by nicely-behaved group schemes. The results of this section will be used in \S\ref{ch:artin} and \S\ref{sec:quasi-split-stacks}.

Let $G$ be a finitely presented group scheme over a base scheme $X$. Recall that the \emph{connected component of unity}, $G^0$, is defined\cite[Exp. VI(B), Def. 3.1]{SGA3} as the functor that assigns to any morphism $S \to X$, the set
$$G^0 (S) = \{ u\in G(S): \forall x \in X, u_x (S_x) \subset G_x^0\}.$$
Here $G_x^0$ is the connected component of unity of the algebraic group $G_x = G \tensor_X \kappa (x)$. By \cite[Exp. VI(B), Thm. 3.10]{SGA3}, if $G$ is smooth over $X$, this functor is representable by a unique open subgroup scheme of $G$. Also note that in this case, $G^0$ is smooth and finitely presented and is preserved by base change \cite[Exp. VI(B), Prop. 3.3]{SGA3}. 

When $G$ is finitely presented and smooth, the quotient space $G/G^0$ exists as a finitely presented and \'etale algebraic space over $X$. For sheaf theoretic reasons, the formation of this quotient is also preserved by base change. 
$G^0$ is not closed in general (interesting examples can be found in \cite[\S 7.3 (iii)]{RaynaudGrpSch} and \cite[Exp. XIX, \S 5]{SGA3}), however we have the following
\begin{lem}\label{lem:closed-conn-comp-unity} Let $G \to X$ be a smooth group scheme and assume that $\bar G= G/ G^0$ is finite and \'etale. Then $G^0$ is a closed subscheme of $G$. 
\end{lem}
\begin{proof} $\bar G \to X$ is finite, hence proper and consequently universally closed. Thus in the cartesian diagram 
$$\diag{\cartesian G^0 \times_X G \ar[r]^{\phi} \ar[d] & G \times_X G \ar[d] \\ \bar G \ar[r] & \bar G \times_X \bar G}$$
the morphism $\phi: (h, g) \mapsto (hg, g)$ is a closed immersion. The property of being a closed immersion is local in the fppf topology and thus by the cartesian diagram
$$\diag{\cartesian G^0 \times_X G \ar[r] \ar[d] & G \times_X G \ar[d] \\ G^0 \ar[r] & G}$$
the embedding of $G^0$ in $G$ is also a closed immersion. The vertical right hand arrow is described by $(g, h)\mapsto g h^{-1}$. 
\end{proof}
Now from Zariski's main theorem (in particular \cite[Lem. 03I1]{SP}) it follows that:

\begin{cor}\label{cor:closed-conn-comp-unity} Let $G$ be a smooth finitely presented group scheme over $X$. There exists a nonempty open scheme $U$ of $X$, such that $G_U^0$ is closed and $G_U/G_U^0$ is finite and \'etale over $U$.  
\end{cor}

\subsection{Centre and centralizers} 

We refer the reader to  \cite[\S 2.2]{Conrad} for a definition of \emph{functorial centralizer}, $Z_G (Y)$, for a closed subscheme $Y$ of a group scheme $G \to X$. It is not generally true that these functors are representable by schemes. However:

\begin{cor}\label{cor:center-exists-after-strat} Let $G$ be a smooth group scheme over an integral base scheme $X$. There exists a nonempty Zariski open $U$ in $X$, such that $G_U$ has a closed subscheme representing its centre. 
\end{cor}

\begin{proof}
By Corollary \ref{cor:closed-conn-comp-unity} we may assume that $G^0$ is a closed and open connected subscheme and $\bar G= G/ G^0$ is a finite \'etale group scheme over $X$. 
Let $\tilde X \to X$ be a finite \'etale cover such that $G|_{\tilde X}$ is a constant finite group over $\tilde X$; i.e. a union of connected components all of which are isomorphically mapped to $\tilde X$. Since 
$G \to \bar G$ is a torsor for a connected group, so is $G|_{\tilde X} \to \tilde X$. Therefore the connected components of the source and target correspond bijectively. But every connected component of $\bar G|_{\tilde X}$ maps isomorphically to $\tilde X$, thus each connected component of $G|_{\tilde X}$ is isomorphic to $G^0$. Now by \cite[Lem. 2.2.4]{Conrad}, the centralizer of each connected component exists over $\tilde X$ and their intersection is the centre of $G|_{\tilde X}$. Finally $\tilde X \to X$ is \'etale and in particular an fpqc covering and the centre of $G|_{\tilde X}$ is affine over $\tilde X$, so affine descent finishes the proof.
\end{proof}

\subsection{Groups of multiplicative type} 
\label{sec:max-tori}

\subsubsection{Quasi-split tori}

A commutative group scheme $T \to X$ is said to be of multiplicative type if it is locally diagonalizable over $X$ in the fppf topology (and therefore in the \'etale topology \cite[Exp. X, Cor. 4.5]{SGA3}). For the general theory of group schemes of multiplicative type we refer the reader to \cite[Ch. IIIV--X]{SGA3}, however we recall a few preliminary facts here. Associated to $T$ there exists \cite[Exp. X, Prop. 1.1]{SGA3} a locally constant \'etale abelian sheaf 
$$T \mapsto M= \uHom_{X-\text{gp}} (T, \mb G_m),$$
and $T$ is the scheme representing the sheaf $\uHom_{X-\text{gp}} (M, \mb G_m)$. This is an anti-equivalence of the categories of $X$-group schemes of multiplicative type and locally constant \'etale sheaves on $X$ whose geometric fibers are finitely generated abelian groups.  

Since every \'etale morphism is Zariski locally finite \cite[Lem. 03I1]{SP},
by shrinking $X$ to a dense open subset of it, we may assume that $T$ is \emph{isotrivial}; i.e. isomorphic to $H \times \mb G_m^r$ after base change along a finite \'etale cover of $X' \to X$. Here $H$ is a constant finite commutative group. Equivalently, $T_{X'}$ is isomorphic to $\uHom_{X'-\text{gp}} (M_{X'}, \mb G_{m, X'})$ where $M$ is a finitely generated abelian group. 

We may also assume that $X' \to X$ is connected and Galois by \cite[Proposition 6.18]{Puttick}). Let $\Gamma$ be the associated Galois group. The action of $\Gamma$ on $X'$ induces an action of it on $M_{X'}$. 

An $X$-torus, $T$, is an $X$-group scheme which is fppf locally isomorphic to $\mb G_{m, X}^r$. This is equivalent to asking for $M$ to be torsion-free. We say $T$ \emph{splits} over $X'$ if $T_{X'} = \mb G^r_{m, X'}$. Our anti-equivalence of categories is now between isotrivial $X$-tori that split over $X'$ and $\Gamma$-lattices (i.e. a finitely generated torsion free abelian groups equipped with the structure of some $\Gamma$-module of finite type) given by 
$$T \mapsto H^0 (X', \uHom_{X'-\text{gp}} (X' \times_X T , \mb G_m))$$
and by 
$$ A \mapsto \uHom_{X-\text{gp}} (X' \times A / \Gamma, T)$$ 
in the reserve direction where $A$ is a $\Gamma$-lattice. In this case, $\chi_T= M_{X'}$ is called the \emph{character lattice} of $T$. 
$T$ is called a \emph{quasi-split} torus if $\chi_T$ is a permutation $\Gamma$-lattice (i.e. the action of $\Gamma$ on $\chi_T$ is by permutation of the elements of a $\mb Z$-basis). 

\subsubsection{Maximal tori of group schemes} 

We recall \cite[Expos\'e IXX, Def. 1.5]{SGA3} that the \emph{reductive rank}, of algebraic $k$-group $G$, is the rank of a maximal torus $T$ of $G_{\bar k}$ where $\bar k$ is the algebraic closure of $k$: 
$\rho_r (G) = \dim_{\bar k} T.$
Likewise, the unipotent rank of $G$ is the dimension of the unipotent radical $U$ of $G_{\bar k}$ and denoted by
$\rho_u (G) = \dim_{\bar k} U.$

For an affine smooth $X$-group scheme $G$, the above integers can be more generally considered as functions on $X$ that assign to every point $x \in X$, the corresponding 
$$\rho_r (x) = \rho_r (G_x), \quad \text{and} \quad \rho_u (x) = \rho_u (G_x).$$ 
The function $\rho_r$ is lower semi-continuous in the Zariski topology \cite[Expos\'e XII, Thm. 1.7]{SGA3}. Moreover, the condition of $\rho_r$ being a locally constant function in the Zariski topology, is equivalent to existence of a global maximal $X$-torus for $G$ in the \'etale topology by \cite[Expos\'e XII, Thm. 1.7]{SGA3}. If $G$ is commutative, then this is furthermore equivalent to existence of a global maximal $X$-torus for $G$ in the Zariski topology \cite[Expos\'e XII, Cor. 1.15]{SGA3}. This immediately implies the following 

\begin{prop}\label{spread-out-maximal-torus} Let $G$ be an affine smooth group scheme over a noetherian base scheme $X$. Then there exists a stratification of $X$ by finitely many locally closed Zariski subschemes $\{ X_i\}$ such that each group scheme $G|_{X_i}$ admits an isotrivial maximal torus. If $G$ is moreover commutative, $G|_{X_i}$ admits a maximal torus in Zariski topology. 
\end{prop} 

\begin{proof} Since $\rho_r$ is lower semi-continuous and integer valued there exists a stratification by locally closed subspaces on which $\rho_r$ is constant. By our earlier discussion, we may further assume each  global torus is isotrivial.
\end{proof} 
 
\subsection{Commutative group schemes}

An affine group scheme $U \to X$ is said to be unipotent if it is unipotent over each geometric fibre. The goal of this section is to prove the following structure decomposition for a commutative group scheme. 
\begin{prop}\label{cor:commutative-structure-spread-out} 
Let $X$ be an integral scheme and $G$ a finitely presented smooth affine commutative group scheme over $X$. Let $\eta$ be the generic point of $X$. Then the decomposition of $G_\eta$ as $T_\eta \times U_\eta$ to a maximal torus and the unipotent radical spreads out; i.e. there exists a non-empty Zariski open $V \subseteq X$ such that  $G_V$ is isomorphic to $T_V \times U_V$ where $T_V$ is a maximal torus with $T_\eta$ as generic fiber and $U_V$ is a unipotent $V$-group scheme with $U_\eta$ as generic fiber. Moreover, if $T_\eta$ is quasi-split we may assume $T_V$ is also quasi-split. 
\end{prop}

We say a group scheme $H_\eta \to X$ (or a property of it with respect to another group scheme $G_\eta$) \emph{spreads out} to a neighborhood of the generic point $\eta \in X$ if 
there exists a dense open subset $U \subset X$ over which there is a $U$-group scheme $H|_U$ pulling back to the prior one 
(and satisfying the same property with respect to a spreading out $G|_U$ of $G_\eta$). 

\begin{lem}\label{lem:space-spreading-out} Let $X$ be an integral scheme and $G$ a finitely presented affine group scheme over $X$. Then closed subgroups of the generic fiber spread out: i.e. 
let $\eta$ be the generic point of $X$ and $H_\eta$ a closed subgroup of $G_\eta$. Then there exists a non-empty open 
set $U \subseteq X$ such that $G|_U$ contains a subgroup scheme $H_U$ fitting in the commutative diagram 
$$\diag{H_\eta \ar[d] \ar[r]& H_U \ar[d]\\ G_\eta \ar[r] & G_U }$$
where the horizontal arrows are pull-back morphisms and the vertical arrows are monomorphisms of group schemes. 
\end{lem}

\begin{proof} It suffices to consider the case where $X$ is an affine scheme $X= \Spec R$. Let $K$ be the function field 
of $R$ with the canonical homomorphism $R \to K$ corresponding to the inclusion of the generic point $\eta \to X$. Then $G = \Spec S$, where $S= R[x_1, \ldots, x_k]/ \mf I$ 
for some finitely generated ideal $\mf I= \langle p_1, \ldots, p_\ell\rangle \subset R[x_1, \ldots, x_k]$ and therefore $G_\eta = \Spec K \tensor_R S$. 
With this notation, $H_\eta$ is cut out as a subscheme by a finitely generated ideal $\mf p \subseteq K \tensor_R S = K[x_1, \ldots, x_k]/ \mf I K$. Thus 
each generator of $\mf p$ can be considered as a polynomial with coefficients in $K$. Since $K$ is the inverse limit of localizations of $R$ in its 
elements, there exists $f \in R$ such that all elements of $\mf p$ are defined with coefficients in $R_f$. This defines a subscheme $H_U$
of $G_U$ satisfying the commutativity of the above diagram, if we set $U = \Spec R_f$. 

Now we put a group scheme structure on $H_U$ by shrinking $U$ further. Let $i: G \to G$ and $m: G \times_X G \to G$, respectively be the inversion and multiplication morphisms on $G$. Considering the inversion morphism, existence of group structure on $H_\eta$ means that in level of coordinate rings, we are given a commutative diagram 
$$\diag{ R_f[x_1, \ldots, x_k] / \mf I R_f \ar^{i^\#}[r]\ar[d] & R_f[x_1, \ldots, x_k] / \mf I R_f  \ar[d]\\
K[x_1, \ldots, x_k]/  \mf I K \ar[r]^{\tilde{i^\#}} & K[x_1, \ldots, x_k] / \mf I K \ar[r]^q & K[x_1, \ldots, x_k] / \mf p \mf I K}$$
and the composition of the induced morphism $\tilde{i^\#}$ and the quotient map $q$ has precisely $\mf p \mf I K$ as its kernel. Hence by a similar argument, there exists some $g \in R_f$ lifting this composition as in the cartesian diagram
$$\diag{R_{fg}[x_1, \ldots, x_k]/ \mf p \mf I R_{fg} \ar[r]\ar[d]& R_{fg} [x_1, \ldots, x_k]/ \mf p \mf I R_{fg}\ar[d]\\
K[x_1, \ldots, x_k]/ \mf p \mf I K \ar[r]& K[x_1, \ldots, x_k]/ \mf p \mf I K.}$$ 
The case of multiplication morphism is similar. So by shrinking further we may assume that $H_U$ is a $U$-group scheme. 
Commutativity of the diagram
$$\diag{ H_U \times H_U \ar[r] \ar@^{(->}[d] & H_U  \ar@^{(->}[d]\\ G_U \times G_U \ar[r] & G_U} $$
where the horizontal arrows are morphisms $(x, y) \mapsto x y^{-1}$ is now obvious. Thus $H_U$ is the desired subgroup scheme of $G_U$.
\end{proof}

\begin{lem}\label{lem:group-homomorphisms-spreading-out} Group homomorphisms (repectively isomorphisms) spread out. Let $G \to X$ and $\eta \in X$ be as in previous lemma. If $G' \to X$ is another group scheme and $\phi_\eta: G_\eta \to G'_\eta$ is a group scheme homomorphism (resp. isomorphism), then there exists a non-empty $U \subseteq X$ and a homomorphism (resp. isomorphism) $\phi: G'_U \to G_U$ such that $\phi|_\eta = \phi_\eta$. 
\end{lem}
\begin{proof} The proof is by similar arguments as in previous lemma.
\end{proof}

\begin{lem}\label{lem:quasi-split-tori-spread-out} The property of being a quasi-split torus spreads out. Let $G \to X$ and $\eta \in X$ be as in the previous lemmas. If $G_\eta$ is a quasi-split torus, then there exists a non-empty $U \subseteq X$ such that $G_U$ is a quasi-split torus. 
\end{lem}
\begin{proof} We may assume that $G$ is also isotrivial. We recall that the character lattice of a torus is expressed in terms of the (\'etale locally constant) sheaf
$\chi(G)= \uHom_{U-\text{gp}} (G, \mb G_{m, U}).$
Restriction from $U$ to $\eta$ induces a homomorphism of finitely generated $\mb Z$-modules 
$$\uHom_{U-\text{gp}} (G, \mb G_{m, U}) \to \uHom_{K-\text{gp}} (G_\eta, \mb G_{m, \eta})$$
and by Lemma \ref{lem:group-homomorphisms-spreading-out} we may assume that this is an isomorphism of $\mb Z$-modules. 
We also note that if $\tilde U \to U$ is a finite \'etale covering that trivializes $G_U$ and restricts to the finite separable extension $L / \kappa (\eta)$, then $\Gamma= \Gal (L /K)$ is at the same time the fundamental group of this covering and the associated action of $\Gamma$ on $\chi(G_\eta)$ induces same action of this group on $\chi(G)$. 
\end{proof}
 
\begin{lem}\label{lem:unipotency-spreading-out} Let $G$ be unipotent group scheme over an integral base scheme $X$. Then there exists a non-empty Zariski open set $U \subset X$ such that $G_U$ has a filtration in subgroups $1 \subset G_1 \subset \ldots \subset G_{r-1} \subset G_r = G$ with all factors $G_{\ell} / G_{\ell -1}$ isomorphic to the constant $U$-group scheme $\mb G_{a, U}$. 
\end{lem}
\begin{proof} Let $H_\eta$ be a subgroup of the generic fiber $G_\eta$. By Lemma \ref{lem:space-spreading-out} the property of being a subgroup spreads out to a non-empty open in $X$. We also observe that the property of being isomorphic to $\mb G_a$ spreads out. That is, if $H_\eta$ is isomorphic to $\mb G_{a, \eta}$, then there exists a non-empty open $U \subseteq X$ such that $H_\eta$ spreads out over it to the constant group scheme $\mb G_{a, U}$. This is another straightforward spreading out argument: let $K$ be the field of fraction of an integral domain $R$. Let $S$ be a finitely presented $R$-algebra generated by $x_1, \ldots, x_\ell$ and that there exists an $R$-algebra isomorphism $\phi: K \tensor_R S \to K[t]$. It is easy to check that there exists a localization $R_f$ of $R$ that extends $\phi$ to an isomorphism $\tilde \phi: R_f \tensor_R S \to R_f [t]$. 
The claim now follows by induction on the quotient scheme $G_U / \mb G_{a, U}$.\footnote{A relevant note is that $\mb A_1$-fibrations are always Zariski locally trivial (cf. \cite{Kambayashi}).}
\end{proof}

The proof of Proposition \ref{cor:commutative-structure-spread-out} is now immediate by spreading the maximal torus of the generic fibre out by Proposition \ref{spread-out-maximal-torus}, and spreading the unipotent radical out by Lemma \ref{lem:unipotency-spreading-out}, and observing that the group structure of $T_\eta \times U_\eta$ also spreads out by Lemma \ref{lem:group-homomorphisms-spreading-out}.

\section{Inertia of groupoids}\label{ch:Gpd}

\newcommand{\gpd}{\mathrm{gpd}}

Let $K(\gpd)$ be the $\mb Q$-vector space generated by finite groupoids,
modulo equivalence and scissor relations.
It is easy to verify that the vector space $K(\gpd)$ is generated by $[BG]$, for finite groups~$G$.

Denote by $\In:K(\gpd)\to K(\gpd)$ the endomorphism sending $[X]$ to
$[\In X]$, where $\In X= X \times_{X \times X} X$ is the inertia groupoid of $X$.  Note that inertia
is compatible with equivalence and scissor relations, so that $\In$ is
well-defined. 

\subsection{Filtration by central order} 

The vector space $K(\gpd)$ has two natural gradings: by order of the
automorphism group, and by order of the centre of the automorphism group. Let $K^{n}(\gpd)$ be the subspace of $K(\gpd)$ generated by those $[BG]$ such that $\# G = n$. Let $K_{i} (\gpd)$ be the subspace generated by $[BG]$ such that $\#Z(G) = i$. Finally $K^n_i(\gpd)$ is generated by those $[BG]$ such that $\#G=n$, and $\#Z(G)=i$.  We have
$$K(\gpd)=\bigoplus_{n=1}^\infty\bigoplus_{i=1}^\infty K^n_i(\gpd)\,.$$
Clearly, $K^{n}(\gpd)$ is finite-dimensional for every $n$, but $K_{i}(\gpd)$ is infinite-dimensional, for all $i$. This grading defines an ascending filtration $K^{\leq n} (\gpd)$ and a descending filtration $K_{\geq i} (\gpd)$. 

\begin{lem}
The endomorphism $\In$ preserves $K^{\leq
  n}(\gpd)$ and $K_{\geq i}(\gpd)$ and on the associated graded,
$K_{\geq i}/K_{> i}(\gpd)$, $\In$ is multiplication by $i$. 
\end{lem}
\begin{proof}
For any finite group $G$, 
$$\In BG\cong\bigsqcup_{g\in C(G)}B Z_G(g)\,,$$
where $C(G)$ denotes the set of conjugacy classes of $G$, and $Z_G(g)$
is the centralizer of $g$ in $G$. Thus,
\begin{align*}
\In [BG]&=\sum_{g\in C(G)}[BZ_G(g)]\\
&=\#Z(G)[BG]+\sum_{g\in C(G)^\ast}[BZ_G(g)]\,,
\end{align*}
where $C(G)^\ast$ denotes the set of non-central conjugacy
classes. Now we note that for non-central $g$ we have strict inequalities
$$\#Z(G)<\#Z_G(g)<\#G\,.$$
This is enought to prove the claim.
\end{proof}

\subsection{Local finiteness and diagonalization} 

\begin{prop}
The endomorphism $\In:K(\gpd)\to K(\gpd)$ is diagonalizable, with spectrum
of eigenvalues equal to the positive integers. 
\end{prop}
\begin{proof}
Every subspace $K^{\leq n}(\gpd)$ is finite dimensional, and preserved by
$\In$. On this finite dimensional subspace, $\In$ is triangular, with
respect to the grading with central order, and with distinct eigenvalues
on the diagonal. This proves that $\In$ is diagonalizable when
restricted to $K^{\leq n}(\gpd)$ for all $n$. 
\end{proof}

Thus $K(\gpd)$ has another natural grading, namely the grading
induced by the direct sum decomposition into eigenspaces under $\In$. This can be interpreted as a grading by the  {\em virtual order of the centre}. Denote the corresponding projection operators by~$\pi_n$. 

For instance, if $A$ is a finite abelian group, then $[BA]$ is an eigenvector for
$\In$, with eigenvalue $\#A$. Thus $\pi_n[BA]=[BA]$, if $A$ had $n$
elements, and $\pi_n[BA]=0$, otherwise.

\begin{ex}
For the dihedral group $D_4$ with eight elements, we have
$$I[BD_4]=2[BD_4]+[B\mb Z_4]+2[BD_2]\,.$$
Hence
$$[BD_4]-\sfrac{1}{2}[B\mb Z_4]-[BD_2]$$
is an eigenvalue of $\In$ with eigenvalue $2$. It follows that
$$\pi_n[BD_4]=\begin{cases}
[BD_4]-\sfrac{1}{2}[B\mb Z_4]-[BD_2]&\text{if $n=2$}\\
\sfrac{1}{2}[B\mb Z_4]+[BD_2]&\text{if $n=4$}\\
0&\text{otherwise}\,.\end{cases}$$
\end{ex}

\subsection{The operators $\In_r$ and eigenprojections}
\label{sec:projection-formula-for-groupoids}

Let $\In_r BG$ be the stack of tuples $(s_1, \ldots, s_r)$ where $s_i$ are $r$ distinct pairwise commuting elements of $G$:
$$\In_r (BG)= [(G^{\times r})^*/ G]\,,$$ 
where the brackets are used as the notation for quotient algebroids. In $K(\gpd)$ we write 
$$\In_r [BG]= [(G^{\times r})^*/ G]\,,$$
where bracket stands for the element in the $K$-group and the
quotient notation is omitted. 

This defines another family of operators on $K(\gpd)$. For $r=0$, $\In_0$ is identity on all $BG$ and $\In_1$ is the usual inertia operator.

\begin{prop}
The operators $\In_r$ for $r\geq0$, preserve the
filtration $K_{\geq k}(\gpd)$ and act as multiplication by 
$r! {k \choose r}$ on the quotient
$K_{\geq k}(\gpd)/K_{> k}(\gpd)$.
\end{prop}

\begin{proof}  Let $n$ be the size of the group $G$ and $k$ the size of its centre. Notice that 
there are $r! {k \choose r}$ ways of choosing a labelled set of $r$ elements from the centre and therefore,
$$\In_r[ BG]=r!\,{k \choose r}\,[BG]+
\sum_{\stack{S\in (G^{\times r})^\ast}{S \not\subseteq Z(G)}}[B Z_G (S)]\,.$$
\end{proof}

\begin{cor}
The operators $\In_r$, for $r\geq0$ are simultaneously
diagonalizable. The common eigenspaces form a family $\Pi_k(\gpd)$ of subspaces of $K(\gpd)$ indexed by positive integers $k\geq0$, and 
$$K(\gpd)=\bigoplus_{k\geq0}\Pi_k(\gpd)\,.$$
Let $\pi_k$ denote the projection onto $\Pi_k(\gpd)$. We have
$$\In_r\pi_k=r!\,{k \choose r}\pi_k\,,$$
for all $r, k\geq0$.
\end{cor}

\begin{cor}
For $r\geq0$, we have
$$\ker \In_r=\bigoplus_{k<r}\Pi_k(\gpd)\,.$$
\end{cor}

\begin{cor}\label{cor:beautiful-invesion}
For every $k\geq0$, we have
$$\pi_k=\sum_{r=k}^\infty \frac{(-1)^{r+k}}{r!} { r \choose k} \,\In_r\,.$$
In
particular, $\pi_0 = \sum_{r=0}^\infty \frac{(-1)^r}{r!} \In_r$, and $\pi_1= \sum_{r=1}^\infty \frac{(-1)^{r-1}}{(r-1)!}\,\In_r\,.$
\end{cor}
\begin{proof} We use the ``beautiful identity'' \cite{MobiusInversion}
$$\sum_{r} (-1)^{r+k} \binom{\ell}{r} \binom{r}{k}  = \delta_{\ell k}$$
to find the projections. We have 
$\id=\sum_{\ell\geq0}\pi_\ell\,,$
and hence
$$\In_r=\sum_{\ell\geq0}\In_r\pi_\ell=\sum_{\ell\geq0}r!\,{\ell \choose r}\,\pi_\ell\,,$$
and therefore
\begin{multline*}
\sum_{r\geq0}\frac{(-1)^{r+k}}{r!} {r\choose k}\,\In_r
=\sum_{r\geq0}\frac{(-1)^{r+k}}{r!} {r \choose k}\left(\sum_{\ell\geq0}r!\,{\ell \choose r}\,\pi_\ell\right)\\
=\sum_{\ell\geq0}\left(\sum_{r\geq0}(-1)^{r+k} {r \choose k} {\ell \choose r}\right)\,\pi_\ell
=\sum_{\ell\geq0}\delta_{\ell,k}\,\pi_\ell=\pi_k\,.
\end{multline*}
\end{proof}

\begin{ex} For the group of permutations of 3 letters, $S_3$ we have
$$\In_r[BS_3]=\begin{cases}
[BS_3] + [B \mb Z_2] + [B \mb Z_3]&\text{if $r=1$} \\
2[B\mb Z_2] + 3 [B \mb Z_3]&\text{if $r=2$} \\
3[B \mb Z_3]&\text{if $r=3$} \\
0&\text{otherwise}\,.
\end{cases}$$
This gives us a way of computing 
$$\pi_n[BS_3]=\begin{cases}
\In_1 - I_2 + \sfrac{1}2 I_3 = [BS_3]-[B\mb Z_2]-\sfrac{1}2 [B\mb Z_3]&\text{if $n=1$}\\
\sfrac{1}2 I_2 - \sfrac{1}2 I_3 = [B\mb Z_2]&\text{if $n=2$}\\
\sfrac{1}6 I_3= \sfrac{1}2 [B\mb Z_3]&\text{if $n=3$}\\
0&\text{otherwise}\,.\end{cases}$$
\end{ex}

\section{Inertia operator on Deligne-Mumford stacks}
\label{ch:DM}

Let $\DM$ be the full subcategory of $\St/B$ of all Deligne-Mumford stacks over $B$. The inertia of a Deligne-Mumford stack is another Deligne-Mumford stack. The Grothendieck ring, $K(\DM)$, has the structure of a $K(\Sch)$-algebra. The inertia endomorphism, respects the scissor relations and is linear with respect to multiplication by schemes. Therefore we have an induced inertia endomorphism on $K(\DM)$. 

\subsection{Irreducible gerbes}

An \emph{irreducible gerbe} is a connected Deligne-Mumford stack, $\XX$, with finite \'etale inertia, $\In \XX \to \XX$.

\begin{prop}\label{prop:strat-DM}Every noetherian Deligne-Mumford stack can be stratified into finitely many locally closed irreducible gerbes.
\end{prop}
\begin{proof} It suffices to show that restricting $\In \XX \to \XX$ to some nonempty open substack of $\XX$ is a finite \'etale morphism. Using \cite[Prop. 5.7.6]{RG} we may assume that $\XX$ is an integral Deligne-Mumford stack. Flatness is an fppf-local property hence by generic flatness \cite[Thm. 6.9.1]{EGA} there exists an open substack of $\XX$ such that $\In \XX \to \XX$ is flat. Since $\In \XX \to \XX$ is unramified, it is \'etale as well. A quasi-finite morphism of schemes is generically finite \cite[Lem. 03I1]{SP}. Therefore by fpqc-descent of finiteness on base \cite[Lem. 02LA]{SP}, $\In \XX \to \XX$ is finite on an open substack of $\XX$. 
\end{proof}
\subsubsection{Central inertia} 

Let $\ZI \XX$ be the full subcategory of $\In \XX$, consisting of tuples $(x, \phi)$ of objects $x$ of $\XX$ and central automorphisms $\phi \in Z(\Aut (x))$. Let $U \to \XX$ be an \'etale cover of $\XX$ by a scheme $U$ such that $\XX_U$ is the neutral gerbe $B_U G$ for an constant finite $U$-group scheme $G$. Then $\ZI \XX|_U$ is isomorphic to the gerbe $B_U Z(G)$. We conclude that $\ZI \XX$ is a Deligne-Mumford stack and $\ZI \XX \to \In \XX$ is a closed immersion. In particular, if $\XX$ is an irreducible gerbe, then $\ZI \XX \to \XX$ is also a finite \'etale cover. 

\begin{rmk} If $\XX$ is an irreducible gerbe, then each connected component $\YY$ of $\In \XX$ is an irreducible gerbe. This is clear since $\YY \to \XX$ is finite \'etale and therefore so is $\In \YY \to \In \XX$ 
by definition of the inertia stack. In other words, an inertia stack of an irreducible gerbe has a canonical stratification into irreducible gerbes by its connected components. 
\end{rmk} 

\subsection{Filtration by split central order}

Recall \cite[Cor. 17.9.3]{EGA} that if $\phi: Y \to X$ is a separated \'etale morphism over a connected base scheme $X$, there is a one-to-one correspondence between the sections of $\phi$ and the number of connected components of $Y$ isomorphic to $X$. Thus for a finite \'etale covering, the number of such sections is an indication of how close $Y$ is to being a trivial 
degree $n$ covering, $\bigsqcup_n X \to X$.

\begin{defn}\label{def:DM-untwistedness} For an irreducible gerbe $\XX$ we define the \emph{split central order}, to be the number of sections of $\ZI \XX \to \XX$. 
\end{defn}
We define an ascending filtration of $K(\DM)$ by declaring $[\XX]$, for an
irreducible gerbe $\XX$, to belong to $K_{\geq n}(\DM)$ if its split central order is at least $n$. 

\begin{prop}\label{prop:DM-filtration-preserved}
The inertia endomorphism on $K(\DM)$ preserves the filtration by split central order.  Furthermore, on the associated graded piece \linebreak $K_{\geq n}(\DM)/ K_{> n}(\DM)$, the inertia endomorphism operates by multiplication by the integer $n$. 
\end{prop}
\begin{proof}
Consider an irreducible gerbe $\XX$, with split central order
$n$ and $\In \XX = \bigsqcup \YY_\alpha$ be the stratification of $\In \XX$ by connected components (hence irreducible gerbes).
There are precisely $n$ of the $\YY_\alpha$ which are contained in $\ZI \XX$ and map isomorphically to $\XX$ (i.e. are degree one connected \'etale covers of $\XX$).
It suffices to show that any other strata $\YY$ has split central number strictly larger than $n$. 

There exists a diagram
$$\vcenter{\xymatrix{
\ZI(\In \XX)\drto_{\pi_3} & \In \XX\times_\XX \ZI\XX \cartesian\rto \dto^{\pi_2}\ar@{_{(}->}_j[l]&
  \ZI\XX\dto^{\pi_1}\\
&\In \XX\rto &\XX}}
$$
where the square is cartesian. For any object $x$ of $\XX$, elements of $\In \XX$ over $x$ are pairs
$(x,\phi)$ such that $\phi \in \Aut(x)$ and objects of $\ZI\XX$ over $x$ are pairs
$(x,\psi)$ where $\psi \in Z(\Aut(x))$.
The fibered product $\In \XX \times_\XX \ZI \XX$ is hence the stack of triples $(x,\phi,\psi)$ with $x, \phi$ and $\psi$ as above. 
On the other hand, $\ZI(\In \XX)$ is the stack of the objects 
$(x,\phi,\psi)$ such that $\phi \in \Aut(x), \psi \in Z(Z_{\Aut(x)}(\phi))$.
Hence there is an embedding of the fibered product into $\ZI(\In \XX)$. Restricting to a substack $\YY \subset \In \XX$ we get the following diagram.
$$\vcenter{\xymatrix{
\ZI \YY \drto_{\pi_3} & \YY \times_\XX \ZI\XX \cartesian\rto \dto^{\pi_2}\ar@{_{(}->}_-j[l]&
  \ZI\XX\dto^{\pi_1}\\
&\YY \rto &\XX\,}}
$$
In this diagram the embedding $j$ is necessarily a union of connected components, because
all vertical and diagonal maps in the diagram are representable finite
\'etale covering maps. Note also that there is a canonical section, $\delta$, to $\pi_3: I(\YY) \to \YY$ via the diagonal $\YY \to \YY \times_\XX \YY$ since any automorphism of an object $x$ in $\XX$ is in its own centralizer. It is obvious that any section of $\pi_1$ pullback to a (distinct) section of $\pi_2$ and gives a (distinct) section of $\pi_3$. This shows that inertia endomorphism preserves $K_{\geq n}$. 

For the action of inertia on the graded piece $K_{\geq n}/ K_{> n}$ we show that if $\YY$ is a component of $\In \XX$ which is not a section of $\pi_1$, then the associated section $\delta$ is not induced by pulling back sections of $\pi_1$. In fact, if $\YY$ is not contained in $\ZI \XX$ then $\delta$ does not lift to $\pi_2$ and we are done. Otherwise, (when $\YY$ is completely contained in $\ZI \XX$), $\delta$ lifts to a section of $\pi_2$ but the image of this section in $\ZI \XX$ is $\YY$ itself, which is not a degree one cover of $\XX$. 
\end{proof}

\subsection{Local finiteness and diagonalization}

In this section we use the notation $\prod_{\XX}^k \YY$ to denote the $k$-fold fiber product of a stack $\YY$ by itself over $\XX$.
We use the notation $I^{(k)} \XX$ for $k$-times application of the inertia construction on the stack $\XX$. We may think of the objects of $I^{(k)} \XX$ as tuples $(x, f_1, \ldots, f_k)$ of an object $x$ in $\XX$ and pairwise commuting automorphisms $f_1, \ldots, f_k$. A morphism $(x, f_1, \ldots, f_k) \to (y, g_1, \ldots, g_k)$ is an arrow $h: x \to y$ of $\XX$ satisfying $h \o f_i = g_i \o h$ for all $i= 1, \ldots, k$. 

\begin{lem}\label{lem:int-cov} Let $\YY \to \XX$ be a finite \'etale representable morphism of algebraic stacks. Then the following family is a finite set up to isomorphism of stacks.
$$\mathcal C(\YY \to \XX) = \{ \WW: \WW \text{ is a connected component of } \prod_\XX^k \YY \text{ for some } k \geq 0\}$$ 
\end{lem}

\begin{proof} 
This is trivial since the Galois closure of $\YY$ with respect to $\XX$ is 
a finite \'etale $\XX$-stack $\overline\YY \to\XX$. And every element in the above family is isomorphic to an intermediate cover, in between $\overline
\YY$ and $\YY$. 
\end{proof}

\begin{cor}\label{cor:int-cov-prod} Let $\YY_1, \cdots, \YY_s$ be finitely many algebraic stacks, finite \'etale over $\XX$. The the following family is finite up to isomorphism.
$$\{ \WW: \WW \text{ is connected component of } \prod_\XX^{k_1} \YY_1\times_\XX \cdots\times_\XX \prod_\XX^{k_s} \YY_s \text{ for } k_1, \ldots, k_s \geq 0\}$$ 
\end{cor}
\begin{proof} There are $s$ projection maps 
$$p_\ell: \prod_\XX^{k_1} \YY_1\times_\XX \cdots\times_\XX \prod_\XX^{k_s} \YY_s  \to \prod_\XX^{k_\ell} \YY_\ell, \quad \ell = 1, \cdots, s$$
which are all finite \'etale and in particular closed and open. The immersion  
$$i: \prod_\XX^{k_1} \YY_1\times_\XX \cdots\times_\XX \prod_\XX^{k_s} \YY_s \to \prod_\XX^{k_1} \YY_1\times \cdots\times \prod_\XX^{k_s} \YY_s$$
is similarly closed and open.
 Hence $\WW$ is isomorphic to its image $i(\WW)$ which is a connected component of 
$\pi_1 (\WW)\times \cdots\times \pi_s (\WW)$. 
However any fiber product $\WW_1 \times \cdots \times \WW_s$
where $\WW_i \in \mathcal C(\YY_i \to \XX)$ has finitely many connected components and by Lemma \ref{lem:int-cov} there are only finitely many such fiber products.
\end{proof}

\begin{cor} Let $\XX$ be an irreducible gerbe. Then the following family is finite up to isomophism.
$$\{ \WW: \WW \text{ is a connected component of } \In^{(m)} \XX \text{ for some } m \geq 0\}$$ 
\end{cor} 
\begin{proof} For an irreducible gerbe $\In \XX \to \XX$ is finite \'etale, hence closed and open and therefore the inertia stratifies to finitely many connected components $\YY_1, \cdots, \YY_s$ which are finite \'etale over $\XX$. 
In the commutative diagram 
$$\diag@C=1em{ 
\In^{(m)} \XX\ \ar@{^(->}[rr]^j\ar[rd] & & \prod_\XX^m \In \XX \ar[ld] \\
& \In \XX } $$
the downward arrows are finite \'etale and hence so is the inclusion $j$. Consequently $j$ is open and closed, and therefore any connected component of $\In^{(m)} \XX$
is a stratum of some substack 
$$\YY_{i_1} \times_{\XX} \cdots \times_\XX \YY_{i_m} \subset \prod_\XX^m \In \XX$$
for a choice of $i_1, \cdots, i_m\in \{ 1, \cdots, s\}$. The claim now follows from Corollary \ref{cor:int-cov-prod}.
\end{proof}

This completes the proof of our main results for Deligne-Mumford stacks:

\begin{thm}[Local finiteness]\label{thm:local-finiteness-DM} Let $\XX$ be a noetherian Deligne-Mumford $B$-stack and $\{ \XX_i \}_{i \in \msf A}$, the stratification of it by irreducible gerbes. Then the $K(\Sch)$-submodule of $K(\DM)$ generated by the set of motivic classes of all $\XX_i$ and all intermediate Galois covers between $\bar{\In \XX_i} \to \In \XX_i$ is finitely generated, invariant under inertia endomorphism, and contains $[\XX]$. 
\end{thm}

\begin{cor}[Diagonalization]\label{cor:diagonalization-DM} The endomorphism $\In:K(\DM)\to K(\DM)$ is diagonalizable, with eigenvalue spectrum equal to $\mb N$, the set of positive integers.
\end{cor}

\subsection{The operators $\In_r$ and eigenprojections}
\label{sec:projection-formula-for-DM}

Let $\In_r \XX$ be the stack of tuples $(x, s_1, \ldots, s_r)$ where $s_i$ are distinct pairwise commuting automorphisms of $x$. By this we mean that of $x: X \to \XX$ is an $X$-point of $\XX$, and $G= \Aut(x)$ is the $X$-group scheme of automorphisms of $x$, then $s_i$ are sections of $G \to X$ and not any two of them are identical sections. 
This definition applies also to $r=0$. The stack $\In_0 \XX$ is just $\XX$. For $r=1$, $\In_1 \XX$ is the usual inertia. Hence $\In_1$ is diagonalizable with integer eigenvalues.

Note that $\In_r$ is closely related to the $k$-fold inertia operators $I^{(k)}$. It is easy to see that by an inclusion-exclusion argument that they satisfy the following identity,
$$\In_r = \sum_{k=1}^r s(r, k) I^{(k)}\,,$$
where $s(r, k)$ are the signed Stirling number of the first kind.

We use the notation $\ZI_r \XX$ for the substack of $\In_r \XX$ consisting of objects $(x, s_1, \ldots, s_r)$ such that all $s_i$ are in the centre of $\Aut (x)$. The complement locus will be denoted by $\NZI_r \XX$. 

Let $\XX$ be an irreducible gerbe with $\In \XX \to \XX$ an \'etale morphism of degree $n$. Then there exists a Galois covering $\tilde\XX \to \XX$ of $\XX$ such that $\ZI \XX|_{\tilde \XX}$ is a disjoint union of $n$ copies of $\tilde \XX$. So we have 
$$[\ZI_r \XX|_{\tilde \XX}] = r! {n \choose r} [\tilde \XX].$$

We use the notation $\Inj(\ul r, \ul n)$ for the set of injections from a set of cardinality $r$ to a set of of cardinality $n$. So 
$$\#\Inj(\ul r, \ul n)=r! {n \choose r}\,.$$

\begin{thm}
The operators $\In_r$, for all $r\geq0$, preserve the
filtration $K_{\geq k}(\DM)$ by split central number.  On the quotient
$K_{\geq k}(\DM)/K_{>k}(\DM)$, 
the operator $\In_r$ acts as multiplication by $r! {k \choose r}$.
\end{thm}

\begin{proof}  Let $\XX$ be an irreducible gerbe with split central number $k$ and $\ZI \XX \to \XX$ be of degree $n$. Let $\tilde \XX \to \XX$ be a splitting cover for $\ZI \XX \to \XX$ with Galois group $\Gamma$. Hence $\Gamma$ acts on $\ul n$. 
$$\diag{
\ZI\XX|_{\tilde \XX} \ar[r]^{\iota} \ar[d] & \ZI \XX \ar[d] \\
\tilde\XX \ar[r] & \XX}$$
Then $\ZI\XX|_{\tilde \XX}$ is the disjoint union of $n$ copies of $\tilde \XX$ and $\Gamma$ acts on it by permuting these copies. Let us rename the $i$-th copy to $\tilde \YY$ and the image to $\YY$. The integer $i \in \ul n$ is fixed under the action of $\Gamma$ precisely when $\YY\cong \tilde \YY/\Gamma$ via the horizontal morphism. Since $\tilde \XX / \Gamma \cong \XX$, the above happens precisely when $\YY$ is isomorphic to a copy of $\XX$ by the vertical morphism. By Proposition \ref{prop:DM-filtration-preserved} this only is the case if $\YY$ is one of the $k$ copies of $\XX$ contributing to the split central number of $\XX$. Hence the set of fixed points of $\Gamma$ is of size $k$. Also,
\begin{align*}
\tilde \XX \times \Inj(\ul r, \ul n) &\longiso \ZI_r \XX|_{\tilde \XX}\,,
\end{align*}
and the action of $\Gamma$ on $\ul n$ induces an action of it on $\Inj(\ul r, \ul n)$ . A morphism $\phi: \ul r \hra \ul n$ is invariant under this action if every element in the image of $\phi$ is so. Therefore the number of fixed points of $\Inj (\ul r, \ul n)$ is $r! {k \choose r}$. We may hence calculate as follows:
\begin{align*}
\ZI_r[\XX]
&=[ {\tilde \XX} \times_\Gamma \Inj(\ul r, \ul n) ]\\
&=\sum_{\phi\in \Inj(\ul r, \ul n)/\Gamma}[\tilde
  \XX/\Stab_\Gamma\phi]\\
&=\sum_{\phi\in\Inj(\ul r,\ul n)^\Gamma}[\XX]+ 
\sum_{\stack{\phi\in\Inj(\ul r,\ul n)/\Gamma}{\Stab_\Gamma\phi\not=\Gamma}}[\tilde \XX/\Stab_\Gamma\phi] 
\end{align*}
Thus, we conclude,
$$\ZI_r[\XX]=r!\,{k \choose r}\,[\XX]+
\sum_{\stack{\phi\in\Inj(\ul r,\ul
    n)/\Gamma}{\Stab_\Gamma\phi\not=\Gamma}}[\tilde
  \XX/\Stab_\Gamma\phi] \,.$$
Finally note that each intermediate cover $\YY= \tilde \XX / \Stab_\Gamma \phi$ has a strictly larger split central number $k$. In fact, $\ZI \YY = \ZI \XX|_\YY$ so  every section of $\ZI \XX \to \XX$ pulls back to a section of $\ZI \YY \to \YY$ but also $\ZI \YY \to \YY$ has sections induced by $\phi$ that do not descend to $\ZI \XX \to \XX$.

Finally for every irreducible gerbe $\YY \subseteq \NZI_r \XX$, the split central rank is strictly larger than $n$, because at least one of the sections $s_i$ is noncentral. 
\end{proof}

\begin{cor}
The operators $\In_r$, for $r\geq0$ are simultaneously
diagonalizable. The common eigenspaces form a family $\Pi_k(\DM)$ of
subspaces of $K(\DM)$ indexed by non-negative integers
$k\geq0$, and 
$$K(\DM)=\bigoplus_{k\geq0}\Pi_k(\DM)\,.$$
Let $\pi_k$ denote the projection onto $K^k(\DM)$. We
have
$$\In_r\pi_k=r!\,{k \choose r}\pi_k\,,$$
for all $r\geq0$, $k\geq0$. 
\end{cor}

\begin{cor}
For $r\geq1$, we have
$$\ker \In_r=\bigoplus_{k<r}\Pi_k(\DM)\,.$$
\end{cor}

\begin{cor}
For every $k\geq0$, we have
$$\pi_k=\sum_{r=k}^\infty \frac{(-1)^{r+k}}{r!} { r \choose k} \,\In_r\,.$$
In
particular, $\pi_0 = \sum_{r=0}^\infty \frac{(-1)^r}{r!} \In_r$, and $\pi_1= \sum_{r=1}^\infty \frac{(-1)^{r-1}}{(r-1)!}\,\In_r\,.$
\end{cor}
The proof is similar to that of Corollary \ref{cor:beautiful-invesion}.

\section{Inertia endomorphism of algebraic stacks} 
\label{ch:artin}

In the rest of this paper, we need to work over a field of characteristic zero. So we let $k = \mb C$, we shorten our notation for the category of algebraic stacks $\St/\mb C$ to $\St$, and let the base category be that of the $\mb C$-varieties, denoted by $\Var$.  

\subsection{Central band of a gerbe}

We recall that to any algebraic stack $\XX$ we can associate an fppf coarse moduli sheaf $X$ of isomorphism classes of objects of $\XX$ 
\cite[Rmk. 3.19]{LMB}. The morphism of stacks $\XX \to X$ is always an fppf (in particular, \'etale) gerbe.

\begin{prop}\label{prop:central-inertia} Let $\XX \to X$ be an \'etale gerbe. Then there exists a sheaf of abelian groups $Z \to X$ and a morphism of sheaves of groups $\phi: Z \times_X \XX \to \In \XX$ such that 
\begin{enumerate}
\item For every $s: S \to \XX$, the induced morphism of sheaves of groups $s^* \phi: Z|_S \to \uAut (s)$ identifies $Z|_S$ with the centre of the sheaf of groups $\uAut (s)$; and,
\item The pair $(Z, \phi)$ is unique, up to isomorphism of sheaves of groups over $X$.
\end{enumerate}
\end{prop}

\begin{proof}
This is explained in \cite[Ch. IV, \S 1.5]{Giraud}, specifically refer to \cite[Ch. IV, \S 1.5.3.2]{Giraud} for existence of the sheaf and to \cite[Ch. IV, Cor. 1.5.5]{Giraud} for the properties of it.
\end{proof}
In the above setting, $Z$ is called the \emph{central band} associated to $\XX$ and if it is a scheme we call it the \emph{central group scheme}. The \emph{central inertia} of $\XX$ is defined to be the fiber product 
$$\diag{\cartesian \ZI \XX \ar[r] \ar[d] & \XX \ar[d] \\ Z \ar[r] & X.}$$

\subsubsection*{Discrete central inertia} 

Suppose we are in the case that $Z \to X$ is a group scheme and consider the open subgroup scheme $Z^0$ and the quotient algebraic space $Z / Z^0$ over $X$. 
Pulling back to $\XX$, we define the \emph{connected component of the central inertia}, $\ZoI \XX$ as the sub-group space, and the \emph{discrete central inertia} as the quotient group space, which are respectively given by the following fiber products
$$\vcenter{\diag{ 
\cartesian \ZoI \XX \ar[r] \ar[d] & \XX \ar[d] \\ Z^0 \ar[r] & X}} \qquad \quad \text{and} \qquad \quad \vcenter{\diag{ \cartesian
\DZI \XX \ar[r] \ar[d] & \XX \ar[d] \\ Z / Z^0 \ar[r] & X.}} $$
Here $Z/ Z^0 \to X$ is called the \emph{discrete central band} of $\XX \to X$. 

\subsection{Clear gerbes}
\label{sec:Grothendieck-ring-Art-stacks}

\begin{defn}\label{dfn:preclear gerbe} Let $\XX$ be an algebraic stack of finite type over the associated coarse moduli space $X$. We say $\XX$ is a {\it clear gerbe} over the algebraic space $X$, if all the following conditions are satisfied.  
\begin{enumerate}[label=(\arabic{*})]
\itemsep-2pt
\item $\XX \to X$ is \'etale gerbe with faithfully flat structure morphism of finite type; 
\item the projection $\In \XX \to \XX$ is a representable smooth morphism of finite type; 
\item $X$ is a $k$-variety (i.e. a reduced, separated, $k$-scheme of finite type);
\item the discrete inertia $\D \XX \to \XX$ is finite \'etale morphism;
\item the central band is a smooth commutative $X$-group scheme;
\item the central inertia is a closed substack of the inertia stack; 
\item the discrete central band is an \'etale finite $X$-group scheme;
\item the central band admits a maximal torus.
\end{enumerate}
\end{defn}
We start with a modification of the stratification in \cite[Thm. 11.5]{LMB}.

\begin{thm}\label{thm:S-stratification} Every algebraic stack of finite type with affine diagonal has a stratification into a disjoint union of finitely many locally closed clear gerbes.
\end{thm} 
\begin{proof} 
By replacing generic smoothness instead of generic flatness in the proof of \cite[Thm. 11.5]{LMB} we may assume (C1) and (C2) are already satisfied and the coarse moduli sheaf $X$, is a noetherian algebraic space of finite type and in particular quasi-compact and quasi-separated. Now we stratify $X$ into $k$-varieties by \cite[Thm. 3.1.1]{CLO}.
It now suffices to show that there exists a Zariski open $U \subseteq X$ such that $\XX|_U$ satisfies conditions (4)--(8). 

Let $\tilde X \to X$ be an \'etale cover of $X$ trivializing $\XX$ to $B_{\tilde X} G$ for a $\tilde X$-group scheme $G$. 
By generic smoothness we may shrink $X$ such that $G$ is smooth over $\tilde X$ and consequently (\cite[Cor. 17.7.3]{EGA}), over $X$. 
By Corollary \ref{cor:closed-conn-comp-unity} we may now assume that (4) is satisfied. 

Let $Z$ be the central band of $\XX \to X$. It is easy to check that the diagonal morphism $Z \to Z \times_X Z$ is representable. 
By Corollary \ref{cor:center-exists-after-strat} we may further assume that $G$ admits its centre $Z(G)$ as a closed subscheme. Then $Z(G)$ is an \'etale cover of the central band and therefore the latter is an algebraic space. Since all stabilizers are affine, we conclude that $Z$ is a scheme. This takes care of (5). By descent of closed immersions (6) is also satisfied. Applying Corollary \ref{cor:closed-conn-comp-unity}, this time to $Z \to X$, will assure (7). 

Recall that a smooth commutative algebraic group over a perfect field has a decomposition $T \times U$ where $U$ is the unipotent radical of the algebraic group and $T$ is a group of multiplicative type. If the group is connected then so is $T$, in which case $T$ is a torus \cite[XIV, Theorem 2.6]{AGS}. Let $\eta$ be a generic point of $X$. Then the mentioned decomposition holds for the generic fiber $Z_\eta$ of central band over $X$. We now use the machinery of \S \ref{sec:max-tori} on how the structure of a commutative algebraic groups spread out by Corollary \ref{cor:commutative-structure-spread-out}.
\end{proof}

\newcommand{\corank}{\mr{ corank} } 

\begin{defn} Let $\XX \to X$ be a clear gerbe. The relative dimensions $\dim_\XX (\In \XX)$ and $\dim_\XX (\ZI \XX)$ are well-defined. The former is called the \emph{total rank}. The latter is called the \emph{central rank} of $\XX$, $\rho(\XX)$, and is bounded above by the former. It is also the rank of the associated commutative group scheme $Z^0 \to X$. 

The \emph{split central degree} is defined to be the number of sections of $\DZI \XX \to \XX$ and is denoted by $\nu (\XX)$. This is bounded above by the \emph{total degree}, defined as the degree of finite \'etale map $\D \XX \to \XX$.

The unipotent and reductive ranks of the central band are constant over the coarse moduli space. We will denote them respectively by $\rho_u (\XX)$ and $\rho_r (\XX)$. \end{defn}

\subsection{Filtration by central rank and split central degree}\label{sec:inertia-endo-Art}  

Let $K_{\geq (r, n)}$ be the subspace of $K(\St)$ spanned by clear gerbes $\XX \to X$, for which the central rank, is at least $r$ and if this rank is exactly $r$ then the split central order is at least $n$. We will now show that the inertia preserves this filtration. First we need a lemma. If $u: G \to H$ is any homomorphism of group schemes, it follows from the definition of the functor of connectedness component of identity \cite[Exp. VI, Def. 3.1]{SGA3} that $G^0$
maps to $H^0$. Moreover we have:

\begin{lem}\label{lem:induced-map-of-components-of-identity} Let $u: G \to H$ be a closed immersion of smooth group schemes of same dimensions over a field $k$, then $u^0: G^0 \to H^0$ is an isomorphism of connected group schemes and the induced quotient map $G/G^0 \to H/H^0$ is a closed immersion of finite group schemes.
\end{lem}
\begin{proof} From the fact that $G^0$ and $H^0$ are irreducible (\cite[Tag 0B7Q]{SP}) we conclude that $G^0 \to H^0$ is a surjective closed immersion. Being a surjective closed immersions is equivalent to being a thickening, but both $G^0$ and $H^0$ are reduced. Therefore $G^0 \to H^0$ is an isomorphism. 
\end{proof}

Let $\XX$ is a clear gerbe. We fix a stratification of $\In \XX$ by clear gerbes and let $\YY$ be one such a stratum. We prove that if $\YY$ is not contained in $\ZI \XX$ then either $\rho(\YY) > \rho(\XX)$ or $\nu(\YY) > \nu(\XX)$. And if $\YY$ is contained in $\ZI \XX$, and maps to a component of 
$\DZI\XX$, {\em not} of degree 1 over $\XX$, then $\rho(\XX)= \rho(\YY)$ but $\nu(\YY) > \nu(\XX)$.

\begin{prop}
\label{prop:descending-noncentral-stratum} 
Let $\YY$ be a stratum of $\In \XX$ not (completely) contained in $\ZI \XX$. 
Then $\rho(\YY) \geq \rho(\XX)$ and if $\rho(\YY) = \rho(\XX)$ then $\nu(\YY) > \nu(\XX)$. 
\end{prop}
\begin{proof} Consider the diagram below, where over every object $(x, \phi\in \Aut(x))$ in $\YY$, $j$ maps $Z(\Aut(x))$ to $Z(Z_{\Aut(x)} (\phi))$. By \cite[Lem. 2.2.4]{Conrad}, $j$ is a closed immersion of group objects on $\YY$.
\begin{equation}\label{diag:central-inertia-of-inertia} 
\vcenter{\xymatrix{
\ZI \YY \drto_{\pi_3} & \YY \times_\XX \ZI\XX \cartesian\rto \dto^{\pi_2}\ar@{_{(}->}_-j[l]&
  \ZI\XX\dto^{\pi_1}\\
&\YY \rto &\XX\,}}
\end{equation}
From this diagram, it is obvious that $\rho(\YY) \geq \rho(\XX)$. 

If $\rho(\YY) = \rho(\XX)$, by Lemma \ref{lem:induced-map-of-components-of-identity} and the fact that the formation of the connectedness component of identity and the quotient by it, are preserved by base change, we have another commutative diagram:
\begin{equation}\label{diag:whtd-discrete}
\vcenter{\xymatrix{
\DZI\YY\drto_{\pi_3} &\ (\DZI \XX)_\YY \ar@{_{(}->}[l]_-{\bar j}\dto^{\pi_2}\rto  & \DZI\XX\dto^{\pi_1}\\
&\YY\rto & \XX}}
\end{equation}
The morphism $\pi_3$ has a canonical section induced by the diagonal morphism $\YY \to \YY \times_\XX\YY$ (explicitly $(x, \phi) \mapsto (x, \phi, [\phi])$, where $x$ is an object of $\XX$, $\phi\in \Aut (x)$ and $[\phi]$ is the orbit of $\phi$ by the action of connectedness component of unity).  
Since $\YY$ is not contained in $\ZI\XX$, this section is not in the image of $\bar j$. Therefore $\nu(\YY) > \nu(\XX)$.
\end{proof}

When $\XX$ is a clear gerbe, the connectedness components of the central inertia $\ZI \XX \to \XX$ are also clear gerbes; this yields a canonical stratification of the central inertia by clears gerbes. The next two propositions pertain to these strata. 

\begin{prop}\label{prop:central-components-are-torsors}  The connectedness components of the central inertia $\ZI \XX \to \XX$ satisfy the following properties: 
(1) they all have the same central rank as that of $\XX$;
and (2) there is a one-to-one correspondence between connectedness components of $\ZI \XX$ and that of $\DZI \XX$ and each compotent $\YY \subseteq \ZI \XX$ is a $(\ZoI\XX)$-torsor over its associated connected component $\YY' \subseteq \DZI \XX$.
\end{prop} 

\begin{proof} 
It is easy to verify the isomorphism $\ZI (\ZI \XX) \cong \ZI \XX \times_{\XX} \ZI \XX$ of stacks which also fit in the commutative diagram  $$\diag{\cartesian \ZI (\ZI \XX) \ar[r] \ar[d] & \ZI \XX \ar[d] \\ Z \times_X Z \ar[r] & Z}$$ Thus for any locally closed substack $\YY$ is $\ZI \XX$ which descends to a locally closed subspace $Y$ of $Z$, we have \begin{equation} \label{diag:inertia-of-commutative-gerbe} \vcenter{\diag{\cartesian \ZI (\YY) \ar[r] \ar[d] & \YY \ar[d] \\  Z \times_X Y \ar[r] & Y.}} \end{equation}In particular the central $Y$-group scheme associated to $\YY$ is the pull-back $Z|_Y$. 
(1) is now obvious form diagram \ref{diag:inertia-of-commutative-gerbe}.
For (2) notice that the morphism  
$\ZI\XX\longrightarrow \DZI\XX$
is a principal bundle for the connected group scheme $\ZoI\XX$
over $\XX$, and therefore there is a bijection between connectedness components of the source
and the target of this morphism. Passing to a component gives us the cartesian diagram 
$$\diag{
\YY \ar[r] \ar@{^(->}[d] & \YY' \ar@{^(->}[d] \\ \ZI \XX \ar[r] & \DZI \XX} $$
together with a finite \'etale mapping $\YY' \to \XX$ proving the lemma. 
\end{proof}

\begin{prop}\label{prop:descending-central-stratum} Let $\YY$ be a connected component of $\ZI \XX$. We always have $\nu(\YY) \geq \nu(\XX)$, with equality happening if and only if the image of $\YY$ in $\DZI \XX$ maps down isomorphically to $\XX$. \end{prop}
\begin{proof} By last lemma, $\YY$ sits over a connectedness component $\YY' \subseteq \DZI \XX$. The homomorphism of commutative group schemes $(\ZI \XX)_{\YY} \to \ZI \YY$ over $\YY$
is an isomorphism giving the left cartesian square of homomorphisms of commutative group schemes and inducing the right hand one:
\begin{equation}\label{diag:ics}
\vcenter{\xymatrix{\cartesian
\ZI\YY\rto \dto_{\pi_2} & \ZI\XX\dto^{\pi_1}\\
\YY\rto & \XX}}
\quad\qquad \vcenter{\xymatrix{
\DZI \YY\ar[r] \ar[d]_{\bar\pi_2}\cartesian & \DZI\XX\dto^{\bar\pi_1}\\
\YY \rto & \XX}}
\end{equation}
So distinct sections of $\bar \pi_1$ pull back to distinct sections of $\bar \pi_2$, therefore $\nu(\YY) \geq \nu(\XX)$.

Now suppose $\YY'$ is not isomorphic to $\XX$. Then, as
the structure map $\YY\to\XX$ factors through 
$\YY'$, and yields a section of $\bar\pi_2$ that is not induced by $\bar\pi_1$.
In this case we have $\nu(\YY) > \nu(\XX)$. 
If $\YY'$ maps isomorphically to $\XX$, the
structure map $\YY\to \XX$ is an $\ZoI\XX$-torsor over $\XX$.
Hence, the upper horizontal map in the left hand diagram of (\ref{diag:ics}) is a torsor for
a connected group scheme, and therefore we can push the sections forward. In this case, $\nu(\YY) = \nu(\XX)$.
\end{proof}

\begin{rmk}\label{rmk:torsor-descends} The $\ZoI \XX$-torsors $\YY$ in Proposition \ref{prop:central-components-are-torsors} and \ref{prop:descending-central-stratum} all come from scheme torsors. In fact, the $\ZoI \XX$-principal bundle $\ZI \XX \to \DZI \XX$ is the pull-back of the $Z^0$-principal bundle $Z \to Z/ Z^0$. Passing to a strata $\YY$, we likewise observe that $\YY \to \YY'$ is the pull back of a $Z^0$-torsor. 
\end{rmk} 

\subsubsection{Filtration by total rank and total degree} 
\label{sec:ascending-filtration} 

Let $K^{\leq d}$ be the span of clear gerbes $\XX \to X$, for which the total rank is at most $d$, and $K^{\leq (d, \delta)}$ be the span of clear gerbes $\XX \to X$, for which the total rank is at most $d$ and if it is exactly $d$ then the total degree is at most $\delta$. 

\begin{prop}
\label{prop:ascending-filtration} 
The inertia operator $\In$ preserves $K^{\leq d}$ and $K^{\leq (d, \delta)}$. 
\end{prop}
\begin{proof} Let $\YY$ be a clear gerbe which is a locally closed substack of $\In \XX$. Consider the commutative diagram 
\begin{equation}\label{diag:total-inertia} 
\vcenter{\diag{
\In \YY \ar[rd] \ar@{^(->}[r]^-i & \In \XX|_\YY \ar[d] \ar[r] & \In \XX \ar[d] \\
& \YY  \ar[r] & \XX}}
\end{equation}
where $i$ is given over every object $(x, \phi \in Aut(x))$ by the closed immersion $Z_{\Aut (x)} (\phi) \hra \Aut (x)$. This shows that $\YY$ does not have a total rank larger than that of $\XX$. If the two total ranks are the same, the by Lemma \ref{lem:induced-map-of-components-of-identity} we have another commutative diagram 
\begin{equation}\label{diag:total-discrete-inertia} 
\vcenter{\diag{
\D \YY \ar[rd] \ar@{^(->}[r]^-i & \D \XX|_\YY \ar[d] \ar[r] & \D \XX \ar[d] \\
& \YY  \ar[r] & \XX}}
\end{equation}
which finished the proof. 
\end{proof}

\subsection{Semisimple and unipotent inertia} 
\label{sec:semisimple-inertia-definition}

If $G$ is an affine group scheme of finite type on base scheme $S$, an element $g \in G(S)$ is defined to be semisimple (unipotent) if for all scheme points $s \in S$ given as spectrum of a field, $g_s$ is semisimple (unipotent) in $G_s$.

\begin{defn} We define the \emph{semisimple (unipotent) inertia} of an algebraic stack $\XX$ to be the strictly full subcategory $\Iss \XX$ ($\Iu \XX$) of the inertia stack $\In \XX$ consisting over a base $S$ of those objects $(x, \phi)$ such that $\phi \in \Aut (x)$ is a semisimple (unipotent) element of the $S$-group scheme $\Aut (x)$. \end{defn} 
It is easy to check that $\Iu \XX \subset \In \XX$ is a closed substack of the inertia. 
According to \cite[3.5.1]{LMB} in order to check that $\Iss \XX$ is a substack of $\In \XX$ we only need to observe that if $f: U \to S$ is an \'etale surjection and $\phi \in \Aut (x)$ is an automorphism of an object $x$ over $S$ where $f^* \phi \in \Aut (f^* x)$ is semisimple then $f$ is also semisimple. But this follows from the above definition and the fact that being semisimple is preserved along field extensions. That is, given a group scheme $G \to S$, if $g$ is a $k$-valued point of $G_k$, and $K/k$ is an algebraically closed extension, and $g'$ the $K$-valued point of $G_K$ obtained by pullback, then $g$ is semisimple if and only if $g'$ is \cite[Expos\'e XII]{SGA3}.

Let $u: U \to \XX$ be an \'etale covering of $\XX$ by an algebraic space. We note that $\Aut^{ss} (u)$ may fail to be an algebraic space, however it is a locally constructible space by \cite[Expos\'e XII, Proposition 8.1]{SGA3}. On the other hand, the diagonal of $\Iss \XX$ is easily seen to be representable, separated and quasi-compact. Therefore $\Iss \XX$ can be written as a well-defined element in $K(\St)$ even though it is not necessarily an algebraic stack. 

A smooth commutative algebraic group $Z$ over a perfect field has a decomposition $Z= Z^{ss} \times U$ where $U$ is the unipotent radical of $Z$ and $Z^{ss}$ is a group of multiplicative type. If $Z$ is connected then so is $Z^{ss}$, in which case the latter is a torus \cite[XIV]{AGS}. The locus determined by the semisimple central automorphisms of objects of $\XX$ is denoted by $\ZssI \XX$ and fits in the cartesian diagram 
$$\diag{ 
\ZssI \XX \ar[r] \ar[d] & \XX \ar[d] \\
Z^{ss}=Z/U \ar[r] & X}$$
and hence has the structure of a group of multiplicative type over $\XX$. We similarly define $\ZuI \XX$ as the pull-back of $U$ along $\XX \to X$. It is easy to check that 
$$\ZuI \XX= \ZI \XX \cap \Iu \XX, \quad \text{ and } \quad 
\ZssI \XX = \ZI \XX  \cap \Iss\XX.$$

\subsection{Spectrum of the unipotent inertia}

Let $K_{\geq u}$ be the sub-vector space of $K(\St)$ spanned by clear gerbes $\XX \to X$, for which the unipotent central rank, is at least $u$. 
We use the notation $K^{\leq d}_{\geq u} = K_{\geq u} \cap K^{\leq d}$. 
 
\begin{prop} The operator $\Iu$ preserves $K^{\leq d}_{\geq u}$ and on the graded piece $K^{\leq d}_{\geq u}/ K^{\leq d}_{> u}$ operates by multiplication by $q^u$. 
\end{prop} 

\begin{proof} Let $\YY$ be a clear gerbe which is a locally closed substack of the clear gerbe $\XX$ with total rank at most $d$ and unipotent central rank at least $u$. Diagram \ref{diag:total-inertia} is enough to conclude that $K^{\leq d}$ is preserved by the unipotent inertia. Also in diagram \ref{diag:central-inertia-of-inertia} all unipotent automorphisms are mapped to unipotent automorphisms therefore inducing another commutative diagram
$${\xymatrix{
\ZuI \YY \drto_{\pi_3} & \YY \times_\XX \ZuI\XX \cartesian\rto \dto^{\pi_2}\ar@{_{(}->}_-j[l]&
  \ZuI\XX\dto^{\pi_1}\\
&\YY \rto &\XX}}
$$
In the level of group schemes, the unipotent radical of the central group scheme of $\XX$, pulls back to a subgroup of the unipotent radical of the central group scheme of $\YY$. This shows that $K_{\geq u}$ is preserved under $\Iu$. However if $\YY$ is not contained in $\ZI \XX$, 
$\pi_3$ has a trivial section that does not lift to $\pi_2$, this means that $\YY$ will have a strictly higher rank unipotent central group scheme. 
We conclude that 
\begin{align*}
[\Iu \XX] &= [\ZuI \XX]   \mod K^{\leq d}_{> u}\\
&= q^u [\XX] \mod K^{\leq d}_{>u}
\end{align*}
The second identity follows from the fact that $U \times_X \XX \to \XX$ is a Zariski locally trivial fibration by Hilbert's Theorem 90.   
\end{proof}

In immediately follows that,

\begin{cor}\label{cor:unipotent-inerta} The endomorphism $\Iu : K(\St) \to K(\St)$ is locally finite and diagonalizable $\mb Z [q]$-module operator on $K(\St)[q^{-1}, \{q^k - 1: k \geq 1\}]$ with eigenvalue spectrum of it consisting of all monomials $q^u$ for $u \geq 0$. 
\end{cor} 

\begin{proof} Local finiteness follows inductively from 
\begin{align*}
\Iu K^{\leq d}_{\geq u} &= q^u K^{\leq d}_{\geq u} + K^{\leq d}_{\geq u+1}\end{align*}
and the fact that $K^{\leq d}_{\geq u}$ is empty for all $u > d$. Diagonalizability is clear from the previous Proposition. 
\end{proof}

\section{Quasi-split stacks}\label{sec:quasi-split-stacks}
\label{ch:quasi-split} 

In this section we present a criteria that if satisfied guarantees the inertia endomorphism is locally finite and diagonalizable. 
For some preliminaries on group schemes of multiplicative type and unipotent group schemes we refer the reader to \S \ref{sec:max-tori}.

\begin{defn} A clear gerbe $\XX \to X$ is called quasi-split it the maximal torus of the central group scheme of it is a quasi-split torus. A full subcategory $\QS$ of $\St$ is called a quasi-split category if it is closed under inertia, scissor relations and fibre products and every object of it has a stratification by quasi-split gerbes.
\end{defn} 

So there is a well-defined induced $K(\Var)$-linear inertia endomorphism on the algebra $K(\QS)$.

\subsection{Motivic classes of quasi-split tori} 
\label{sec:motivic-computation-by-induction} 

Let $\Gamma$ be a finite group acting on the finite set $\ul r= \{1, 2, \ldots, r\}$. We consider an integer partition $\lambda \vdash r$ with declaring $\lambda_i$ to be the number of orbits of size $i$ and define a polynomial 
$$\ms Q_{\lambda} = \prod_{i=1}^\ell (q^{i} - 1)^{\lambda_i}\,.$$

\begin{prop}\label{prop:motivic-computation}\label{prop:motivic-class-of-quasi-split-torus} Let $T$ be an isotrivial quasi-split torus over the integral scheme $X$ and 
$\bar X$ a splitting cover of it, with Galois group 
$\Gamma$ acting by permutation of a basis of $\chi (T_{\bar X})$. The motivic class of $T$ is given by
\begin{align}\label{eq:motive-formula}
[T]= \ms Q_{\lambda} (q)[X]
+\sum_{\stack{I_\bullet \in F(\ul r)/\Gamma}{\Stab_\Gamma(I)
    \subsetneq\Gamma}}(-1)^{\ell(I_\bullet)}q^{|I_{\max}|}[\bar 
  X/\Stab_\Gamma(I)]\,.
\end{align}
where $\lambda \vdash r$ is the partition associated to the action of $\Gamma$ on a basis of $\chi (T)$. 
\end{prop}

\begin{proof} For a subset $I\subseteq \ul r$, we denote by $\mb A^I\subseteq A^r$ the
subset of all $(x_1,\ldots,x_r)$ such that $x_i=0$, for $i\not\in I$,
and by $\mb G_m^I\subseteq \mb A^I$, the set of all $(x_1,\ldots,x_r)\in
\mb A^I$, such that $x_i\not=0$, for all $i\in I$.  For every
$I\subseteq \ul r$, we have a $\Gamma$-equivariant stratification 
$$\mb A^I=\bigsqcup_{J\subseteq I}\mb G_m^J\,,$$
and hence, 
$$[\mb G_m^I]=[\mb A^I]-\sum_{J\subsetneq I}[\mb G_m^J]\,.$$
By induction, we get an 
equivariant inclusion-exclusion principle
$$[\mb G_m^r]=\sum_{k\geq0}(-1)^k\sum_{I_\bullet\in F^k(\ul r)}[\mb A^{I_{k}}]\,.$$
Here $F^k(\ul r)$ is the set of all flags $I_\bullet=(I_k\subsetneq\ldots
\subsetneq I_1\subsetneq I_0=\ul r)$.  We denote the length of a flag
$I_\bullet$ by $k=\ell(I_\bullet)$, the maximal index by $k=\max$, and
the set of all flags, regardless of their lengths, by $F(\ul r)$. 

$\Gamma$ acts on the split torus $T_{\bar X}= \Spec (\mc O_{\bar X} [\chi_T])$  and by \cite[Prop. 5.21]{Lenstra} we can revive $T$ from this pull-back as 
$T \cong (\bar X \times_X T)/ \Gamma$. Then the surjection of $\mb Z$-module, 
$\bigoplus_{i=1}^r b_i \mb Z \to \chi_T$,
induces a sheaf homomorphism $\mc O_{\bar X} [b_1,\cdots, b_r] \to \mc O_{\bar X} [\chi_T]$ and consequently an open immersion 
$\mb G_{m, \bar X}^r \cong T_{\bar X}\to \mb A^r_{\bar X}$,
which is equivariant for the $\Gamma$-action. Hence we may pass to quotient schemes and get 
$$T = {\bar X} \times_\Gamma \mb G_{m}^r \hra {\bar X} \times_\Gamma \mb A^r\,.$$
We have
\begin{align*}
[T]&=[\bar X\times_\Gamma \mb G_m^r]\\
&=\sum_{k\geq0}(-1)^k[\bar
  X\times_\Gamma\bigsqcup_{I_\bullet\in F^k(\ul r)}\mb A^{I_k}] \\ 
&=\sum_{k\geq0}(-1)^k\sum_{I_\bullet\in F^k(\ul r)/\Gamma}[\bar
  X\times_{\Stab_\Gamma(I_\bullet)}\mb A^{I_k}]\\
&=\sum_{I_\bullet\in F(\ul r)/\Gamma}(-1)^{\ell(I_\bullet)}q^{|I_{\max}|}[\bar
  X/\Stab_\Gamma(I_\bullet)]\\
&=\sum_{I_\bullet\in F(\ul r)^\Gamma}(-1)^{\ell(I_\bullet)}q^{|I_{\max}|}[X]\\*
&\phantom{mmmm}+\sum_{\stack{I_\bullet\in F(\ul n)/\Gamma}{\Stab_\Gamma(I)
    \subsetneq\Gamma}}(-1)^{\ell(I_\bullet)}q^{|I_{\max}|}[\bar 
  X/\Stab_\Gamma(I)]\\
&= \ms Q_{\Omega} (q)[X]
+\sum_{\stack{I_\bullet \in F(\ul r)/\Gamma}{\Stab_\Gamma(I)
    \subsetneq\Gamma}}(-1)^{\ell(I_\bullet)}q^{|I_{\max}|}[\bar 
  X/\Stab_\Gamma(I)]\,.
\end{align*}
Note that all forms of affine spaces occurring in this computation are vector bundles over their base by Hilbert's Theorem 90. This is the reason for the appearance of the
terms  $q^{|I_k|}$ in the calculation. 
\end{proof} 

\subsection{Spectrum of the inertia}
\label{sec:descending-twist-type}

Let $\XX\to X$ be a quasi-split clear gerbe with a central $X$-group scheme $Z$ of reductive rank $t$. 
The Galois group $\Gamma= \pi_1 (\bar X/X)$ of the minimal splitting Galois cover is then a subgroup of $S_t$, the group of permutations of $t$ letters. 
This action induces an integer partition $\lambda \vdash t$ as explained in 
\S\ref{sec:motivic-computation-by-induction}.

\begin{defn} For a quasi-split clear gerbe $\XX$, the partition $\lambda$ constructed as above is called the \emph{twist type} of it. 
\end{defn}
We impose a well-ordering on the set of all integer partitions: 
\begin{enumerate} \itemsep-2pt
\item if $\lambda \vdash t$ and $\mu \vdash s$ and $t < s$ then $\lambda < \mu$; and 
\item if $\lambda, \mu \vdash t$ and $b(\lambda) < b(\mu)$ then $\lambda < \mu$. 
\end{enumerate}
Here $b(\lambda) = \sum_i \lambda_i$ is the \emph{number of blocks} of the partition $\lambda$. If $s = t$ and $b(\lambda) = b(\mu)$ the ordering between $\lambda$ and $\mu$ is not important. We may consider the lexicographic ordering for this case. 

The double-filtration in section \ref{sec:inertia-endo-Art} is similarly preserved in $K(\QS)$. We now introduce a refinement, $K_{\geq (\cdot, \cdot, \cdot)}$ of this filtration by declaring $K_{\geq (r, n, \lambda)}$ to be generated by those quasi-split clear gerbes that have central rank at least $r$, and if they have central rank exactly $r$, then their split central degree is at least $n$, and if it is exactly $n$, then their twist type is at least $\lambda$.

\begin{lem}\label{lem:QS-triangularization} The inertia endomorphism of $K(\QS)$ respects the filtration $K_{\geq(\cdot, \cdot, \cdot)}$ and on each graded piece 
$K_{\geq(r, n, \gamma)}/K_{>(r, n, \gamma)}$ 
operates by multiplication by the polynomial $n q^{r-t} \ms Q_{\gamma \vdash t} (q).$
\end{lem}
\begin{proof} In view of results of section \ref{sec:inertia-endo-Art} it suffices to consider a tuple $(r, n, \lambda)$, and a quasi-split clear gerbe $\XX \to X$ with central group scheme $Z \to X$ of rank $r$, split central degree $n$, and twist type $\lambda \vdash t$ where $t$ is the rank of the maximal torus $T \subseteq Z$. Also from Propositions \ref{prop:descending-noncentral-stratum} and \ref{prop:central-components-are-torsors} we only need to consider a central stratum $\YY$ of $\In \XX$ which is an $\ZoI \XX$-torsor over $\XX$ and descends over $X$ to a $Z^0$-torsor $Y$. There are $n$ such strata. 

Let $\eta \in X$ be the generic point with residue field $K$. Over the algebraic closure, we have a decomposition $Z_{\bar K} = T_{\bar K} \times U_{\bar K}$ where $U_{\bar K}$ is the unipotent radical. By \cite[Cor. 15.10]{MilneALA}, this direct product decomposition descends to $K$ and by Lemma \ref{lem:unipotency-spreading-out} and Corollary \ref{cor:commutative-structure-spread-out} spreads out to a non-empty open of $X$ which we may without loss of generality assume is $X$ itself.
Both $U_X$ and $T_X$ are special (\cite[Prop 2.2]{Bergh}) and therefore so is $Z^0$. By Proposition \ref{prop:motivic-class-of-quasi-split-torus}  
$$
[Y]= q^{r-t} \ms Q_{\lambda \vdash t} (q) [X] + \sum_{\stack{I_\bullet \in F(\ul r)/\Gamma}{\Stab_\Gamma(I)
    \subsetneq\Gamma}}(-1)^{\ell(I_\bullet)}q^{r-t+ |I_{\max}|}[\bar X/\Stab_\Gamma(I)]\,,
$$    
for a finite \'etale covering $\bar X \to X$. 
And pulling back along $\XX \to X$,
\begin{equation}\label{eq:torsor-formula}
[\YY]= q^{r-t} \ms Q_{\lambda \vdash t} (q) [\XX] + \sum_{\stack{I_\bullet \in F(\ul r)/\Gamma}{\Stab_\Gamma(I)
    \subsetneq\Gamma}}(-1)^{\ell(I_\bullet)}q^{r-t+ |I_{\max}|}[\bar \XX/\Stab_\Gamma(I)]\,,
\end{equation}
where $\bar \XX = \XX|_{\bar X}$. 
Note that each of these stacks is a clear gerbe over the intermediate cover $\bar X / \Stab I$ and the maximal torus of $\XX$ pulls back to the maximal torus $T|_{\bar X / \Stab I}$ of the central band of $\bar \XX/ \Stab_\Gamma (I)$. Since $\Gamma'= \Stab I$ is a proper subgroup of $\Gamma$, the orbit space $\ul r/ \Gamma'$ has strictly more elements than $\ul r / \Gamma$. Therefore the twist type of $\bar \XX/ \Stab_\Gamma (I)$ is strictly larger than that of $\XX$.
\end{proof}

\begin{prop} The endomorphism $I: K(\QS) \to K(\QS)$ is locally finite. For every element $\XX$ in $\QS$ there exists a finite-dimensional $\mb Z[q]$-module of $K(\QS)$ that is invariant under inertia and contains the motivic class $[\XX]$. 
\end{prop}

\begin{proof} Recall the ascending filtration $K^{\leq (d, \delta)}$ of $\St$ by total rank and total degree. Similar filtration is well-defined on $\QS$. There are finitely many possible twist types for the maximal tori of clear gerbes in $K^{\leq (d, \delta)}$. The central rank is bounded above by $d$, and the finest twist type is the partition of $d$ by $d$ blocks, which we denote by $1_{d} \vdash d$ and we may write 
$$K^{\leq (d, \delta)} = K^{\leq (d, \delta, 1_d)}, \quad K^{\leq (d, \delta, 1_d)}_{\geq (r,n, t)}=  K^{\leq (d, \delta, 1_d)} \cap  K_{\geq (r,n, t)}.$$
By a similar argument as of Proposition \ref{prop:ascending-filtration}, $K^{\leq (d, \delta, 1_d)}$ is preserved by the inertia. By the previous proposition we have 
$$\In K^{\leq (d, \delta, 1_d)}_{\geq (r, n, \lambda)} = n q^{r-t} \ms Q_{\gamma \vdash t} (q) . K^{\leq (d, \delta, 1_d)}_{\geq (r, n, \lambda)} + K^{\leq (d, \delta, 1_d)}_{\geq (r, n, \lambda)}$$
and the claim follows by induction. For the base case of this induction we only need to observe that $K^{\leq (d, \delta, 1_d)}_{\geq (r, n, \lambda)}$ is empty if either $r > d$ or $n > d$ and a clear gerbe in 
$K^{\leq (d, \delta, 1_d)}_{\geq (d, \delta, 1_d)}$ 
consists of $\delta$ strata each of which is a torsor for the split torus $\mb G_m^d$. Therefore the inertia operator acts on this filtered piece by multiplication by $\delta (q-1)^d$.  
\end{proof}
It immediately follows that,
\begin{thm}\label{thm:main}
The operator $\In$ is diagonalizable on $K(\QS)(q)= \mb Q(q) \tensor_{\mb Q[q]} K(\QS)$ as a linear endomorphism of a $\mb Q(q)$-vector space.  
The eigenvalue spectrum of it is the set of all polynomials of the form 
$$n q^u \prod_{i=1}^k (q^{r_i} -1).$$
\end{thm}

\begin{rmk} In fact $\In$ is diagonalizable as an endomorphism of the $\mb Z[q]$-module 
$$K(\QS)[q^{-1}, \{ (\ms Q_\lambda - \ms Q_\mu)^{-1}: \forall \lambda \vdash t, \mu \vdash s\}).$$
\end{rmk}

\subsection{Spectrum of the semisimple inertia}

We can now prove a semisimple version of Proposition \ref{prop:descending-noncentral-stratum} in terms of the reductive ranks of the quasi-split clear gerbes rather than their central ranks. Note that the finite group scheme $Z/ Z^0$ is semisimple and product of semisimple commuting elements is semisimple. Therefore there are $\nu (\XX)$ connected components of $\DZssI \XX$ that map isomorphically to $\XX$. In fact, the morphism $Z^{ss} / Z^{ss, 0} \to Z/ Z^0$ is an isomorphism of finite $X$-group schemes.

\begin{prop}
Let $\YY$ be a stratum of $\Iss \XX$ not completely contained in $\ZssI \XX$. 
Then $\rho_r(\YY) \geq \rho_r(\XX)$ and if $\rho_r(\YY) = \rho_r(\XX)$ then $\nu(\YY) > \nu(\XX)$. 
\end{prop}
\begin{proof} 
Let $\YY \subseteq \Iss \XX$ be a strata of the semisimple inertia, in particular a locally closed substack of $\In \XX$. Diagram \ref{diag:central-inertia-of-inertia} has a semisimple verison. The downward arrows $\pi_1$, $\pi_2$ and $\pi_3$ are all structure morphisms of relative commutative group schemes. Since unipotency is preserved under the group homomorphisms we may divide each of these commutative group schemes with their unipotent radical. 
\begin{equation}
\vcenter{\xymatrix{
\ZssI \YY \drto_{\pi_3} & \YY \times_\XX \ZssI\XX \cartesian\rto \dto^{\pi_2}\ar@{_{(}->}_-j[l]&
  \ZssI\XX\dto^{\pi_1}\\
&\YY \rto &\XX}}
\end{equation}
It now follows that $\rho_r (\YY) \geq \rho_r(\XX)$. If $\pi_2$ and $\pi_3$ are of identical relative dimensions, then we may pass to the discrete central semisimple inertia similar to the argument of  Proposition \ref{prop:descending-noncentral-stratum}:
\begin{equation}\label{diag:whtd-discrete}
\vcenter{\xymatrix{
\DZssI \YY\drto_{\pi_3} &\ (\DZssI \XX)_\YY \ar@{_{(}->}[l]_-{\bar j}\dto^{\pi_2}\rto  & \DZssI\XX\dto^{\pi_1}\\
&\YY\rto & \XX}}
\end{equation}Same analysis as in case of Proposition \ref{prop:descending-noncentral-stratum} shows that $\pi_3$ has strictly more sections that $\pi_1$. 
\end{proof}

When $\XX$ is a quasi-split clear gerbe, the connectedness components of the semisimple central inertia $\ZssI \XX \to \XX$ are also quasi-split clear gerbes; this yields a canonical stratification of $\ZssI \XX$. The analogue of Propositions \ref{prop:central-components-are-torsors} and \ref{prop:descending-central-stratum} is stated below: 

\begin{prop} Let $\YY$ be a connectedness component of $\ZssI \XX \to \XX$. Then 
$\rho_r(\YY) = \rho_r(\XX)$. We always have $\nu(\YY) \geq \nu(\XX)$, with equality happening if and only if the image of $\YY$ in $\DZssI \XX$ maps down isomorphically to $\XX$.
\end{prop} 

\begin{proof} The proof is similar to that of Propositions \ref{prop:central-components-are-torsors} and \ref{prop:descending-central-stratum} by considering the commutative diagram 
$$\diag{\cartesian \ZssI (\ZssI \XX) \ar[r] \ar[d] & \ZssI \XX \ar[d] \ar[r] & \XX \ar[d] \\ 
Z^{ss} \times_X Z^{ss} \ar[r] & Z^{ss} \ar[r] & X}$$
for the first claim where by restricting to $\YY$ we get 
$$\diag{\cartesian \ZssI (\YY) \ar[r] \ar[d] & \YY \ar[d] \\  Z^{ss} \times_X Y \ar[r] & Y.}
$$
Note that $\ZssI \XX\longrightarrow \DZssI\XX$ is a principal bundle for the connected group scheme $\ZssoI\XX$ over $\XX$, and therefore there is a bijection between connectedness components of the source
and the target of this morphism. So $\YY$ sits over a connectedness component $\YY' \subseteq \DZI \XX$. Similar to the case in Proposition \ref{prop:descending-central-stratum} we now can form the following cartesian diagram
$$\diag{\DZssI \YY\ar[r] \ar[d]_{\bar\pi_2}\cartesian & \DZssI\XX\dto^{\bar\pi_1}\\
\YY \rto & \XX}
$$
and the rest of the proof is now similar to Proposition \ref{prop:descending-central-stratum}. 
\end{proof}

The proofs of local finiteness and diagonalization of $\Iss$ now follow similar to the case of the full inertia operator. The filtration to consider in this case is $K^{\leq (d, \delta, 1_d)}_{\geq (t, n, \lambda)}$ where $K_{\geq (t, n, \lambda)}$ is generated by those clear gerbes for which the reductive rank is at least $t$, and if it is exactly $t$ then the discrete central order is at least $n$ and if it is exactly $n$, then the twist type of the maximal torus is at least $\lambda$. The semisimple inertia acts on the filtered pieces via 
$$\In^\ss K^{\leq (d, \delta, 1_d)}_{\geq (r, n, \lambda)} = n \ms Q_{\gamma \vdash t} (q) . K^{\leq (d, \delta, 1_d)}_{\geq (r, n, \lambda)} + K^{\leq (d, \delta, 1_d)}_{\geq (r, n, \lambda)}.$$
We conclude that,

\begin{thm}\label{thm:semisimple-main}
The endomorphism $\Iss: K(\QS) \to K(\QS)$ is locally finite and triangularizable and 
the eigenvalue spectrum of it is the set of all polynomials of the form 
$$n \prod_{i=1}^k (q^{r_i} -1).$$
Moreover $\Iss$ is a diagonalizable $\mb Q(q)$-linear endomorphism of the vector space $K(\QS)(q)$. 
\end{thm}

\section{Examples}

\begin{ex}\label{ex:BGL2} A first simple example is the case of $[\BGL_2]$. Here, and in following examples 
we are suppressing the notation $[.]$ for quotient stacks; thus unless mentioned otherwise, all quotients (of schemes) are stack quotients. Note that we have
$$\In \BGL_2= \GL_2/\GL_2= (\GL_2)^{\text{\tiny ss, eq}}/\GL_2
\sqcup (\GL_2)^{\text{\tiny dist}}/\GL_2
\sqcup (\GL_2)^{\text{\tiny ns}}/\GL_2\,.$$
The first stratum contains diagonalizable matrices with one eigenvalue, the second stratum diagonalizable matrices with distinct
eigenvalues, and the third stratum the non-semisimple matrices. We study
these three strata and their inertia: 

\underline{First stratum}: Consider the mapping
$\mb G_{m}\to \GL_2$ via $x\mapsto \smat{x& 0\\ 0 & x}$.
This is equivariant with respect to the natural $\GL_2$-action, so we
get an induced morphism of stacks
$$\mb G_{m}\times \BGL_2\to \GL_2/\GL_2\,$$
which is easily seen to be an isomorphism onto the first stratum.

\underline{Second stratum}: Let $T$ be the standard maximal torus of $\GL_2$. Let
$\Delta$ be the centre of $\GL_2$, which is the diagonal subtorus of
$T$.  Let $N$ be the normalizer of $T$.  We have a short exact
sequence
$$\xymatrix@1{
0\rto & T\rto & N\rto & \mb Z_2\rto & 0}$$
where $\mb Z_2$ is the Weyl group of $\GL_2$. Note that $N=\mb G_m^2\rtimes
\mb Z_2$ is in fact a semi-direct product, by taking $\smat{0&1\\1&0}$ as the nontrivial element of $\mb Z_2\subset N$.
The induced action of the Weyl group $\mb Z_2$ on $T$ is by swapping the two entries. 
The natural inclusion map $T\setminus\Delta\to \GL_2$ is equivariant for the inclusion $N\subset \GL_2$, so we get an induced morphism of stacks
$$(T\setminus \Delta)/N\longrightarrow \GL_2/\GL_2\,$$
which is an isomorphism onto the second stratum. We will abbreviate this as
$$\XX=(T\setminus\Delta)/N.$$

\underline{Third stratum}: Let $H$ be the (commutative) subgroup of all matrices of the form
$\smat{\lambda&\mu\\ 0 & \lambda}$. Note that $H$ is the centralizer of every matrix of the form
$\smat{a&1\\0&a}$, with $a\not=0$. Thus we see that the third stratum
is isomorphic to $\mb G_{m}\times \B H$. 

We conclude that in the level of motivic classes the inertia of the class $[\BGL_2]$ is given by
$$\In [\BGL_2]= (q-1)[\BGL_2] + [\XX] + (q-1)[\B H]\,.$$

Since $H$ is commutative, we also have $\In [\B H] = q(q-1) [\B H]$. We will now find the inertia
of the second stratum $\XX$.  Note 
that the coarse moduli space of $\XX$ is the smooth variety $X = T\setminus \Delta/\mb Z_2$. Note also that $\In \XX = \ZI \XX$ as the stabilizer of any point in $T \setminus \Delta$ is commutative. We will write $\tilde X=T\setminus \Delta$ to emphasize the fact that $T\setminus \Delta $ is a degree 2
cover of $X$. 

Associated to the $\mb Z_2$-action on the group $T$, there exists a commutative $X$-group scheme
$$T'=\tilde X\times_{\mb Z_2} T,$$
with fibre $T$. By Lemma \ref{lem:neutral-gerbe-finder}, $\XX$ is the neutral gerbe, 
$\XX=B_X T'$. And $\ZI \XX = \In \XX$ fits in the cartesian diagram 
$$\xymatrix{
\In \XX\rto \dto & \XX\dto \\
T'\rto & X.}$$

The representation of $\mb Z_2$ on $\mb A^2$ given by swapping entries, yields 
a canonical closed embedding $T'\subset V$, into a rank 2 vector bundle 
over $X$. As in Proposition \ref{prop:motivic-class-of-quasi-split-torus}, this 
leads to
$$\In [\XX] = (q^2 - 1) [\XX]- (q-1)^2 (q-2)[\B \mb G_m^2]$$

Thus, the 4-dimensional $K(\Var)$-module, $L$, of motives generated by the 4
motives 
$$ [\BGL_2]\,, [\B H]\,, [\XX]\,,\text{ and } [\B \mb G_m^2]$$
is preserved by the inertia endomorphism $\In$. The first element in this 
set is of central rank 1 and the other three are of central rank 2. $\B H$ has 
reductive rank $1$, $\XX$ has reductive rank 2 with the nontrivial partition of $2$ associated
to it, and $\B \mb G_m^2$ has a rank 2 torus with the trivial partition of $2$ associated to it. The eigenvalue
spectrum is hence $\{ q-1, q(q-1), q^2 -1, (q-1)^2\}$. Inertia endomorphism is
lower triangularizable on $L$ and we have
$$
\In= \pmat{ q-1 & 0 & 0 & 0 \\
 q-1 & q (q-1) & 0 & 0 \\
 1 & 0 & q^2 - 1& 0\\
 0 & 0 & -(q-1)^2 (q-2) & (q-1)^2 }
$$
with a set of eigenvectors

\begin{table}[h!]
\centering
$$
\begin{array}{r|l}
\text{Eigenvalues} & \text{Eigenvectors}\\ \hline
q-1 & -q(q-1) [\BGL_2] + q [\B H] + [\XX] + (q-1) [\B \mb G_m^2]\\
q (q-1) & [\B H]\\
q^2 - 1 & [\XX] - \frac{(q-1)(q-2)}2 [\B \mb G_m^2]\\
(q-1)^2 & [\B \mb G_m^2]
\end{array}
$$
\caption{Spectrum of the inertia endomorphism on a 4-dimensional $K(\Var)$-submodule of $K(\St)$ containing $[\BGL_2]$}
\label{table:1}
\end{table}
Also $\In$ is diagonalizable on $L[q^{-1}, (q-1)^{-1}]$ and the eigenprojections of $[\BGL_2]$ are 

\begin{table}[h!]
\centering
$$\begin{array}{r|l}
\text{Eigenvalues} & \text{Eigenvectors}\\ \hline
\Pi_{q-1} &  [\BGL_2] - \frac{q}{q-1} [\B H] - \frac{1}{q(q-1)} [\XX] - \frac{1}{q} [\B \mb G_m^2]\\
\Pi_{q (q-1)} & \frac{q}{q-1}[\B H]\\
\Pi_{q^2 - 1} & \frac{1}{q(q-1)} [\XX] - \frac{(q-2)}{2q} [\B \mb G_m^2]\\
\Pi_{(q-1)^2} & \frac{1}2[\B \mb G_m^2]
\end{array}$$
\caption{Eigenprojections of $[\BGL_2]$}
\label{table:1}
\end{table}
\end{ex}

\begin{lem}\label{lem:neutral-gerbe-finder}
Let $\Gamma$ be a group acting on the group $T$ by automorphisms, $\Gamma \to
\Aut(T)$, and let $G=N\rtimes H$ be the associated semi-direct product of
groups. Let $X$ be a variety, $\tilde X\to X$ a principal $\Gamma$-bundle,
and $T' \to X$ is the associated form of $T$ over $X$. The $B_X T' =\tilde X/G$.
\end{lem}
\begin{proof}
Consider the diagram
$$\xymatrix{
\tilde X\times T\rto \dto & \tilde X\times T\dto\\
\tilde X\rto & \tilde X}$$
where $\Gamma$ acts on the first column and $G$ on the second
column in the obvious way.  Then the horizontal arrows are a morphism of $T$-bundles
which is $\Gamma \to G$ equivariant.  Thus we get an induced cartesian
diagram of stacks
$$\vcenter{\xymatrix{
(\tilde X\times T)/\Gamma \rto \dto & (\tilde X\times T)/G\dto\\
\tilde X/H\rto & \tilde X/G}}
 \quad \text{which we may rewrite as } \quad
\vcenter{\xymatrix{
T'\rto \dto & X\dto\\
X \rto & \tilde X/G.}}$$
Then the latter diagram induces a morphism $\XX\to \tilde X/G$, which
is then obviously an isomorphism.
\end{proof}

\begin{ex}\label{ex:T2Z2} In the previous example the central group schemes were always connected. We will now present an example that demonstrates how non-connected central group schemes contribute to non-monic eigenvalues. Let $N = \mb G_m^2 \rtimes \mb Z_2$ be the group scheme introduced in previous example. In this example we study $\B N$. The inertia of $\B N$ has two obvious connectedness components: 
$$\In\, \B N  = N / N  = \mb G_m^2 / N \sqcup \mb G_m^2 \times \{ \sigma\} / N$$
where $\sigma$ is the nontrivial element of $\mb Z_2$. 

\underline{First stratum}: This stratum is not already a gerbe (as the stabilizer of points on diagonal $\Delta \subset \mb G_m^2$ is not isomorphic to the stabilizer of other points). However the following is a stratification of it into clear gerbes:
$$\mb G_m^2 / N = \Delta / N \sqcup \XX,$$
where $\XX = \mb G_m^2 \setminus \Delta / N$ is the same quotient stack that appeared in previous example. The action of $N$ on $\Delta$ is trivial so we have  
$$[\mb G_m^2 / N ] = (q-1) [\B N] + [\XX].$$

\underline{Second stratum}: This stratum is already a clear gerbe and we will denote it as $\YY$. Any point $\langle(\mu, \gamma), \sigma\rangle$ of $\mb G_m^2 \times \{\sigma\}$ is conjugate to $\langle (\mu\gamma, 1), \sigma\rangle$ which is canonical for the orbit. Thus the subscheme $Y$ representing the points,
$$\{ \langle (x, 1), \sigma \rangle \} \subset \mb G_m^2 \times \{\sigma\},$$
is a coarse moduli space for this gerbe. This is isomorphic to $\mb A^1 \setminus \{0\}$. The stabilizer of this subscheme is a subgroup scheme $N' \subseteq N$, the fiber of which over a geometric point $\bar x$ of $Y$ is the $\kappa(\bar x)$-algebraic group
$$N'_{\bar x} = \{ \langle(t, t), 1\rangle: t \in \kappa (\bar s)^\times \} \cup \{ (xt, t), \sigma\rangle: t \in \kappa(\bar s)^\times \}.$$  
We notice that the mapping $Y / N' \to \YY$ is an isomorphism of stacks and also that $N'$ is a commutative $Y$-group scheme, acting trivially on $Y$. Therefore
\begin{align*}
[\YY] &= [Y] [\B N'] = (q-1) [\B N']\\
\In [\B N'] &= [N'/ N'] = [N'] [\B N'] 
\end{align*}
Finally $N'/ (N')^0 \to Y$ is a degree two covering of $X$. The image of $(N')^0$ in $N' / (N')^0$ is isomorphic to $Y$ and therefore so is the image of the other connected component. So $N'$ is Zariski locally the union of two $\mb G_m$-torsors over $Y$. Pulling back along $\YY \to Y$ we have  
$$\In [\B N'] = 2 (q-1)[\B N'].$$

We conclude that the $K(\Var)$-submodule of $K(\St)$ generated by 
$$[\B N]\,, [\B N']\,, [\XX]\,, \text{ and }[\B \mb G_m^2]$$ 
is invariant under inertia endomorphism. The first two generators have central rank one, and $[\B N]$ has split central number one whereas $[\B N']$ has split central number two. The spectrum of $\In$ restricted to this submodule is the set 
$$\{ (q-1), 2 (q-1), q^2 - 1, (q-1)^2\}$$
as expected. 
\end{ex}

\begin{ex}\label{ex:BGL3} Another simple example that shows many features of this theory is the stack $\BGL_3$. As before, the inertia stack is isomorphic to the quotient stack $[\GL_3/ \GL_3]$ via conjugation action of $\GL_3$ on itself. We first stratify this quotient according to Jordan canonical forms: let $J^k_\lambda$ be the subscheme of all general linear matrices with $k$-distinct eigenvalues and $\lambda \vDash 3$ is a partition of $3$ indicating format of the Jordan blocks and $R^k_\lambda \rightrightarrows J^k_\lambda$ is the groupoid representation of restriction of $[\GL_3 / \GL_3]$ to $J^k_\lambda$. Then we have a stratification 
\begin{align*}
[\GL_3/ \GL_3] &= [ J^1_{(3)} / R^1_{(3)}] \sqcup [ J^1_{(2,1)} / R^1_{(2,1)}] \sqcup [ J^1_{(1,1,1)} / R^1_{(1,1,1)}] \\
&\sqcup  [J^2_{(2, 1)} / R^2_{(2, 1)} ] \sqcup  [J^2_{(1,1, 1)} / R^2_{(1,1, 1)}] \\
& \sqcup [J^3_{(1,1,1)} / R^3_{(1,1,1)}  ]
\end{align*}
The action of $R^k_\lambda$ on $J^k_\lambda$ by conjugation is always trivial unless in presence of Jordan blocks of same dimension with distinct eigenvalues (which can then be permuted). Thus
\begin{align*}
[\GL_3/ \GL_3] &= J^1_{(3)} \times \B R^1_{(3)} \sqcup J^1_{(2,1)} \times \B R^1_{(2,1)} \sqcup J^1_{(1,1,1)} \times \B R^1_{(1,1,1)}\\
&\sqcup  J^2_{(2, 1)} \times \B R^2_{(2, 1)}  \sqcup  J^2_{(1,1, 1)} \times \B R^2_{(1,1, 1)} \\
& \sqcup J^3_{(1,1,1)} \times \B R^3_{(1,1,1)}  
\end{align*}
We recall the notation of Example \ref{ex:BGL2} for the subgroup of upper-triangular $2 \times 2$ matrices with a single eigenvalue of multiplicity two: 
$$H = \left\{ \pmat{ a & b \\ 0 & a }: a, b \in \mb G_m\right\}.$$
This represents a commutative group scheme. Now, easy computations show that all $R^k_\lambda$'s are subgroup schemes of $\GL_3$ and in fact
\begin{table}[h!]
\centering
$$\begin{array}{c|c|c}
\text{Groupoid} & \text{Group scheme structure} & \text{Commutative?} \\ \hline
R^1_{(3)}& \left\{ \pmat{  a & b & c \\ 0 & a & b \\ 0 & 0 & a}: a\in \mb G_m, b, c \in \mb A^1 \right\} & \text{Yes} \\
R^1_{(2,1)}& \left\{ \pmat{  a & b & c \\ 0 & a & 0 \\ 0 & d & e}: a, e \in \mb G_m, b, c,d  \in \mb A^1 \right\} & \text{No}  \\
R^1_{(1,1,1)}& \GL_3 & \text{No} \\
R^2_{(2,1)}& H \times \mb G_m & \text{Yes} \\
R^2_{(1,1,1)}& \GL_2 \times \mb G_m & \text{No} \\
R^3_{(1,1,1)}& \mb G_m^3 \rtimes S_3 & \text{Yes}
\end{array}$$
\caption{Stratification of $\GL_3$}
\label{table:1}
\end{table}

\begin{align*}
[\GL_3/ \GL_3] &= (q-1) [ \B R_{(3)}^1] + (q-1) [\B R^1_{(2,1)}] + (q-1) [\BGL_3] \\
& +  (q-1)(q-2) [\B H][\B\mb G_m] +  (q-1)(q-2) [\BGL_2][\B\mb G_m] \\
& + [\mb G_m^3 / \mb G_m^3 \rtimes S_3 ]
\end{align*}
Since inertia respects the commutative algebra structure of $K(\St)$ we may use the previous example to compute the effect of inertia on terms of the second line above. Since $R_{(3)}^1$ is commutative we also have 
$$\In [\B R_{(3)}^1] = [R_{(3)}^1] [\B R_{(3)^1}] = q^2 (q-1) [\B R_{(3)}^1].$$ 

The case of $\YY= [\mb G_m^3 / \mb G_m^3 \rtimes S_3]$ is similar to that of $[\mb G_m^2 / \mb G_m^2 \rtimes \mb Z_2]$. 
It remains to analyze the action of $G = R^1_{(2,1)}$ on itself. We need to stratify $G/G$ to several substacks which is carried out in Table \ref{table:stratification-r1}.
\begin{table}[t]
\centering
$$\begin{array}{c|c|c}
\text{Strata} & \text{Canonical form for an orbit} & \text{Centralizer of the canonical form} \\ 
\hline
\pmat{ a & b & c \\ 0 & a & 0 \\ 0 & d & e }: a \neq e & \pmat{ a & b + cd / (a-e) & 0 \\ 0 & a & 0 \\ 0 & 0 & e } &G_1 = \left\{ \pmat{ x & y & 0 \\ 0 & x & 0 \\ 0 & 0 & z }: x, z \neq 0 \right\} \\
\pmat{ a & b & c \\ 0 & a & 0 \\ 0 & d & a }: c, d \neq 0 & \pmat{ a & 0 & 1 \\ 0 & a & 0 \\ 0 & c d & a } &G_2= \left\{ \pmat{ x & y & z \\ 0 & x & 0 \\ 0 & w & x }: x \neq 0 \right\}\\
\pmat{ a & b & 0 \\ 0 & a & 0 \\ 0 & d & a }: d \neq 0 & \pmat{ a &b + d & 0 \\ 0 &a &0 \\0 &d &a } &G_3 = \left\{ \pmat{ x & y & 0 \\ 0 & x &0 \\ 0 & z & x }: x \neq 0 \right\}\\
\pmat{ a & b & c \\ 0 & a & 0 \\ 0 & 0 & a }: c \neq 0 & \pmat{ a & 0 & 1 \\ 0 & a & 0 \\ 0 & 0 &a } &G_4 = \left\{ \pmat{ x & y & z \\0 &x &0 \\0 & 0&x }: x \neq 0 \right\}\\
\pmat{ a & b & 0 \\ 0 & a & 0 \\ 0 & 0 & a } & \text{itself} & G 
\end{array}$$
\caption{Stratification of $R^1_{(2, 1)}$}
\label{table:stratification-r1}
\end{table}
It follows that 
\begin{multline*}\In [\B G ] = q(q-1) [\B G ] + q^3 (q-1)(q-2) [\B G_1]\\ + 
q (q-1)^3 [ \B G_2] + q (q-1)^2 [\B G_3] + q (q-1)^2 [\B G_4].
\end{multline*}
We conclude that $[\BGL_3]$ is contained in a 9-dimensional $K(\Var)$-submodule of $K(\St)$ which is diagonalizable (Table \ref{table:bgl3-9dim}). 

\begin{table}[h!]
\centering
$$
\begin{array}{c|c|c|l|l}
\text{Central rank} & \text{Reductive rank} & \text{Twist type}&  \text{Pivot elements} & \text{Eigenvalue} \\ \hline\hline
1& 1& (1) & [\BGL_3] & q-1 \\ \hline
2& 2& (2, 0) & [\BGL_2][\B \mb G_m] & (q-1)^2\\\cline{2-5}
& 1& (1)& [\B G] & q (q-1) \\ \hline
3& 3& (0, 0, 1)& [\YY] & q^3 -1\\ \cline{3-5}
 & & (1, 1, 0)& [\XX] [\B \mb G_m] & (q^2-1)(q-1) \\\cline{3-5}
& & (3, 0, 0)& [\B \mb G_m^3] & (q-1)^3\\ \cline{2-5}
& 2& (2, 0)& [\B H][\B\mb G_m], [\B G_1] & q (q-1)^2\\\cline{2-5}
& 1& (1)& [\B R^1_{(3)}], [\B G_3], [\B G_4] & q^2 (q-1)\\ \hline
4& 1& (1)&[\B G_2] & q^3 (q-1)
\end{array}$$
\caption{Spectrum of the inertia endomorphism of a 9-dimensional $K(\Var)$-submodule of $K(\St)$ containing $[\BGL_3]$}
\label{table:bgl3-9dim}
\end{table}
\end{ex}


\begin{thebibliography}{10}

\bibitem{SGA3}
M.~Artin, J.~E. Bertin, M.~Demazure, A.~Grothendieck, P.~Gabriel, M.~Raynaud,
  and J.-P. Serre.
\newblock {\em Sch\'emas en groupes}.
\newblock S\'eminaire de G\'eom\'etrie Alg\'ebrique de l'Institut des Hautes
  \'Etudes Scientifiques. Institut des Hautes \'Etudes Scientifiques, Paris,
  1963/1966.

\bibitem{BRHall}
K.~{Behrend} and P.~Ronagh.
\newblock {The inertia operator on the motivic Hall algebras}.
\newblock {\em ArXiv e-prints}, December 2016.

\bibitem{Bergh}
Daniel Bergh.
\newblock {\em Computations in the Grothendieck Group of Stacks}.
\newblock PhD thesis, Department of Mathematics, Stockholm University, February
  2012.

\bibitem{Conrad}
Brian Conrad.
\newblock Reductive group schemes (sga3 summer school).
\newblock \url{math.stanford.edu/~conrad/papers/luminysga3.pdf}, 2011.
\newblock Retrieved on November 2012.

\bibitem{CLO}
Brian Conrad, Max Lieblich, and Martin Olsson.
\newblock Nagata compactification for algebraic spaces.
\newblock {\em Journal of the Institute of Mathematics of Jussieu},
  11:747--814, 2012.

\bibitem{Giraud}
Jean Giraud.
\newblock {\em Cohomologie non ab\'elienne}.
\newblock Springer-Verlag, Berlin, 1971.
\newblock Die Grundlehren der mathematischen Wissenschaften, Band 179.

\bibitem{EGA}
Alexander Grothendieck and Jean-Alexandre-Eug\`ene Dieudonn\'e.
\newblock {\em \'{E}l\'ements de g\'eom\'etrie alg\'ebrique.}
\newblock Number 4, 8, 11, 17, 20, 24, 28, and 32 in Publications
  Math\'ematiques. Institut des Hautes \'Etudes Scientifiques, 1960-1967.

\bibitem{Kambayashi}
T.~Kambayashi and David Wright.
\newblock Flat families of affine lines are affine-line bundles.
\newblock {\em Illinois J. Math.}, 29(4):672--681, 12 1985.

\bibitem{LMB}
G{\'e}rard Laumon and Laurent Moret-Bailly.
\newblock {\em Champs alg\'ebriques}, volume~39 of {\em Ergebnisse der
  Mathematik und ihrer Grenzgebiete. 3. Folge. A Series of Modern Surveys in
  Mathematics}.
\newblock Springer-Verlag, Berlin, 2000.

\bibitem{Lenstra}
H.~W. Lenstra.
\newblock Galois theory for schemes.
\newblock \url{http://websites.math.leidenuniv.nl/algebra/GSchemes.pdf}, 2012.
\newblock Retrieved on February 2012.

\bibitem{MilneALA}
James~S. Milne.
\newblock Algebraic groups, lie groups, and their arithmetic subgroups.
\newblock \url{www.jmilne.org/math/}, 2011.
\newblock Retrieved on November 2012.

\bibitem{AGS}
James~S. Milne.
\newblock Basic theory of affine group schemes, 2012.
\newblock Available at \url{www.jmilne.org/math/}.

\bibitem{Puttick}
Alexandre Puttick.
\newblock Galois groups and tehe \'etale fundamental group.
\newblock 2012.
\newblock Retrieved on September 2012.

\bibitem{RaynaudGrpSch}
Michel Raynaud.
\newblock {\em Faisceaux amples sur les sch\'emas en groupes et les espaces
  homog\`enes}.
\newblock Lecture Notes in Mathematics, Vol. 119. Springer-Verlag, Berlin,
  1970.

\bibitem{RG}
Michel Raynaud and Laurent Gruson.
\newblock Crit\`eres de platitude et de projectivit\'e. {T}echniques de
  ``platification'' d'un module.
\newblock {\em Invent. Math.}, 13:1--89, 1971.

\bibitem{MobiusInversion}
John Riordan.
\newblock Inverse relations and combinatorial identities.
\newblock {\em The American Mathematical Monthly}, 71(5):485--498, 1964.

\bibitem{SP}
The {Stacks Project Authors}.
\newblock \itshape stacks project.
\newblock \url{http://stacks.math.columbia.edu}, 2015.
\newblock Last time retrieved on January 2016.

\end{thebibliography}
\end{document}